\documentclass[12pt]{article}

\usepackage{bm}
\usepackage{amsbsy}
\usepackage{amsmath}
\usepackage{amssymb}
\usepackage{amsthm}
\usepackage{authblk}
\usepackage{bbm}
\usepackage{color}
\usepackage{enumitem}
\usepackage{geometry}
\usepackage{graphicx}
\usepackage{natbib}
\usepackage[hidelinks]{hyperref}
\usepackage{mathtools}
\usepackage{setspace}
\usepackage{enumerate}
\usepackage{natbib}

\newcommand{\blind}{1}

\theoremstyle{plain}
\newtheorem{theorem}{Theorem}

\newtheorem{corollary}{Corollary}

\newtheorem{remark}{Remark}
\newcommand{\RNum}[1]{\uppercase\expandafter{\romannumeral #1\relax}}

\def\FDP{\mathrm{FDP}}
\def\FDR{\mathrm{FDR}}

\def\Var{\mathrm{Var}}
\def\Cov{\mathrm{Cov}}

\def\bX{\mathbf{X}}
\def\bY{\mathbf{Y}}

\def\bmu{\boldsymbol{\mu}}
\def\bSigma{\boldsymbol{\Sigma}}
\def\calH{\mathcal{H}}

\begin{document}

\def\spacingset#1{\renewcommand{\baselinestretch}%
{#1}\small\normalsize} \spacingset{1}


\if1\blind
{
  \title{\bf Asymptotic Uncertainty of False Discovery Proportion}
  \author{Meng Mei$^\text{a}$, Tao Yu$^\text{b}$, and Yuan Jiang$^\text{a}$\thanks{Yuan Jiang is the corresponding author. Meng Mei and Yuan Jiang's research is supported in part by National Institutes of Health grant R01 GM126549. Tao Yu's research is supported in part by the Singapore Ministry Education Academic Research Tier 1 Funds: R-155-000-202-114.}}
  \affil{$^\text{a}$ Department of Statistics, Oregon State University\\
         $^\text{b}$ Department of Statistics and Data Science, National University of Singapore}
  \date{}
  \maketitle
} \fi

\if0\blind
{
  \bigskip
  \bigskip
  \bigskip
  \begin{center}
    {\LARGE\bf Asymptotic Uncertainty of False Discovery Proportion}
\end{center}
  \medskip
} \fi

\begin{abstract}

Multiple testing has been a popular topic in statistical research. Although vast works have been done, controlling the false discoveries remains a challenging task when the corresponding test statistics are dependent. Various methods have been proposed to estimate the false discovery proportion (FDP) under arbitrary dependence among the test statistics. One of the main ideas is to reduce arbitrary dependence to weak dependence and then to establish theoretically the strong consistency of the FDP and false discovery rate (FDR) under weak dependence. As a consequence, FDPs share the same asymptotic limit in the framework of weak dependence. We observe that the asymptotic variance of the FDP, however, may rely heavily on the dependence structure of the corresponding test statistics even when they are only weakly dependent; and it is of great practical value to quantify this variability, as it can serve as an indicator of the quality of the FDP estimate from the given data. As far as we are aware, the research on this respect is still limited in the literature. In this paper, we first derive the asymptotic expansion of FDP under mild regularity conditions and then examine how the asymptotic variance of FDP varies under different dependence structures both theoretically and numerically. With the observations in this study, we recommend that in a multiple testing performed by an FDP procedure, we may report both the mean and the variance estimates of FDP to enrich the study outcome. 

\end{abstract}

\noindent%
\vfill

\newpage
\spacingset{1.9} 

\section{Introduction}
\label{Sec-1}

Multiple hypothesis testing has been a popular topic in statistical research; it has wide applications in many scientific areas, such as biology, medicine, genetics, neuroscience, and finance. Consider a multiple testing problem where we need to test $p$ hypotheses simultaneously; denote by $p_0$ and $p_1$ the number of true null and false null hypotheses respectively. The testing outcomes can be summarized in Table \ref{Table-1}, 
\begin{table}[h]
\caption{Summary of possible outcomes from multiple hypothesis testing}
\begin{center}
\begin{tabular}{l|cc|c} \hline
 & Not rejected & Rejected & Total \\ \hline
 True Null & $U$ & $V$ & $p_0$ \\
 False Null & $T$ & $S$ & $p_1$ \\ \hline
 Total & $p - R$ & $R$ & $p$ \\ \hline
\end{tabular}
\end{center}
\label{Table-1}
\end{table}
where $V$ and $R$ denote the number of hypotheses that are wrongly rejected (i.e., the number of type I errors made) and rejected (i.e., the number of rejections made) out of the $p$ hypotheses, respectively.

Conventional methods propose to control the familywise error rate (FWER):  $\text{FWER} = P(V > 0)$ \citep{bonferroni1936teoria, vsidak1967rectangular, holm1979simple, simes1986improved, holland1987improved, hochberg1988sharper, rom1990sequentially}, or the generalized familywise error rate (gFWER): $\text{gFWER} = P(V > k), \text{ for } k \ge 1$ \citep{dudoit2004multiple, pollard2004choice, lehmann2012generalizations}. These methods control the probability of zero or a limited number of false positives, therefore they are usually conservative, especially when the number of hypotheses $p$ is large. We observe that large $p$ is common in practice. For example, in a genome-wide association study, hundreds of thousands or even millions of single nucleotide polymorphisms (SNPs) of each subject and the associated disease status are measured. It is of interest to find which SNPs are associated with the disease status. A common approach is to conduct a hypothesis test for each SNP; this results in a huge number of hypothesis tests. 

When $p$ is large, FDP and its statistical characteristics have been introduced. False discovery rate (FDR) has been the most popularly studied; it was introduced by \citet{benjamini1995controlling} and was proven to be a reasonable criterion; in particular,
\begin{equation*}
\FDP = \frac{V}{R}, \qquad \FDR = E(\FDP), 
\end{equation*}
where $\FDP = \FDR = 0$ if $R = 0$. A list of FDR research can be found in \citet{benjamini1995controlling}, \citet{benjamini2001control}, \citet{storey2002direct}, \citet{storey2003statistical}, \citet{storey2004strong}, \citet{ferreira2006benjamini}, \citet{clarke2009robustness}, and the references therein. We observe that most of these methods are founded on the assumption that the test statistics are independent or satisfy some restrictive dependence structure. 

Because of the complicated nature of the practical data, the resultant test statistics may follow any arbitrary dependence structure. The effects of dependence on FDR have been studied extensively, e.g., by \citet{benjamini2001control}, \citet{finner2002multiple}, \citet{owen2005variance}, \citet{sarkar2006false}, \citet{efron2007correlation}, among others. Recently, new methods for the control and/or estimation of FDR have been proposed to allow complex and general dependence structure or even incorporate it in the testing procedure. Examples include the local index of significance (LIS) testing procedure for dependence modeled by a hidden Markov model \citep{sun2009large}, the pointwise and clusterwise analysis for spatial dependence \citep{sun2015false}, and the principal factor approximation (PFA) method for arbitrary dependence \citep{fan2012estimating, fan2017estimation, fan2019farmtest}. In particular, \citet{fan2012estimating} studied FDP and FDR under arbitrary dependence with a two step procedure. They first proposed a PFA method to reduce arbitrary dependence to weak dependence; and then established theoretically the strong consistency of the FDP and FDR. Therefore, their method is valid for test statistics under arbitrary dependence; under weak dependence, both FDP and FDR converge to the same asymptotic limit.

In addition to FDR, another popular multiple testing criterion is called false discovery exceedance (FDX, also referred to as TPPFP in the literature), which is defined as the tail probability of the event that FDP exceeds a threshold $q \in (0, 1)$:
\begin{equation*}
\text{FDX} = P(\FDP > q).
\end{equation*}
Compared to FDR that is only a single-value summary, FDX has the potential to fully describe the uncertainty and/or distribution of FDP as the threshold value $q$ varies. However, it has not been fully characterized how the dependence among the test statistics would impact FDX as opposed to FDR in the literature. In particular, some works on FDX assume independence among tests \citep{genovese2004stochastic, guo2007generalized, ge2012control}; more importantly, most works focus on developing procedures that control FDX as a single-value characteristic of FDP other than investigating how dependence affects the uncertainty and/or distribution of FDP \citep{korn2004controlling, van2004augmentation, lehmann2005generalizations, genovese2006exceedance, delattre2015new, hemerik2019permutation, dohler2020controlling, basu2021empirical}. Thus, these works still lack the valuable information about how FDP varies from study to study.

As far as we are aware, most of the above works have focused on proposing the estimates of FDP and/or its related statistical characteristics
. On the one hand, studying these estimates is sufficient to ensure that the corresponding FDP and/or its statistical characteristics are under control. On the other hand, we observe that the asymptotic uncertainty of the FDP may rely heavily on the dependence structure of the corresponding test statistics; and it is of great practical value to quantify this variability, as it can serve as an indicator of the quality of the FDP estimate and benefit the research of the FDR and FDX methods. More specifically, with the FDP variance estimate available, one can conclude more confidently how well the FDP is controlled when controlling its corresponding mean at a given level. It may also benefit the development on the methods for estimating and controlling FDX. The research on the uncertainty of FDP has high potential impact, but it is still limited in the community. For example, \citet{ge2012control} and \citet{delattre2011false} derived the asymptotic distribution of FDP under independence and Gaussian equi-correlation dependence, respectively; \citet{delattre2016empirical} established the asymptotic distribution of $\FDP$ under some assumptions on the dependence structure of the test statistics and that the expected values of the test statistics in the alternatives are equivalent. 

In this paper, our focus is to provide a rigorous theoretical investigation of the asymptotic uncertainty of FDP in the framework that the test statistics are weakly dependent. In particular, we derive the theoretical expansion of the FDP: it is a linear combination of the tests, based on which we obtain its asymptotic variance. We also examine how this variance varies under different dependence structures among the test statistics both theoretically and numerically. One interesting observation is that the FDP based on weakly dependent test statistics converges to the same limit as that based on independent test statistics \citep{fan2012estimating}; however, the asymptotic variance of the former is almost always significantly greater than the latter both theoretically and numerically. With these observations, we recommend that in a multiple testing performed by an FDP procedure, we may report both the mean and the variance estimates of FDP to enrich the study outcome. Finally, with a real GWAS dataset, we demonstrate how our findings may impact real studies.

\section{Asymptotic Limit of FDP Under Weak Dependence}
\label{Sec-2}

Consider a genome-wide association study with $n$ subjects; for each subject, its disease status and $p$ SNPs are measured. The data can be represented by $\bY = (Y_1, \ldots, Y_n)^T$ and a $n\times p$ matrix $\bX = (X_{ij})_{i=1,\ldots,n; j=1,\ldots,p}$, where $Y_i$ denotes the disease status of subject $i$, and the $i$th row of $\bX$ collects its associated SNPs. When $Y_i$'s are measured in a continuous scale, the association between the $j$th SNP and the disease status can be modelled by the marginal linear regression model:
\begin{equation}
Y_{i} = \alpha_j + \beta_jX_{ij} + \epsilon_{ij}, \quad i=1,\ldots,n, \label{linear.model}
\end{equation}
where $\alpha_j$ and $\beta_j$ are regression parameters and $\epsilon_{ij}$ are random errors. The estimates $\hat \beta_j,\ j=1,\ldots,p$ and the corresponding test statistics $Z_j,\ j=1, \ldots, p$ for
\begin{equation} \label{hypotheses.1}
H_{0j}: \beta_{j}=0 \quad \text{versus} \quad H_{1j}:\beta_{j} \neq 0,\quad j=1,\dots,p,
\end{equation}
can be derived with a standard least squares procedure.  Proposition 1 in \citet{fan2012estimating} verified that $\hat \beta_j,\ j=1,\ldots,p$ are correlated and derived their joint distribution, under appropriate regularity conditions. As a consequence, the test statistics $Z_j,\ j=1,\ldots,p$ are also correlated. In particular, they derived that conditioning on $\bX$,
\begin{equation}\label{test.distribution}
(Z_{1},\ldots,Z_{p})^{T} \sim N((\mu_{1},\ldots,\mu_{p})^{T},\bSigma),
\end{equation}
where $(\mu_1,\ldots, \mu_p)^T$ and $\bSigma = (\sigma_{ij})_{i,j=1,\ldots,p}$ are appropriate population mean and variance matrix for $(Z_1,\ldots,Z_p)^T$. Furthermore, the testing problem \eqref{hypotheses.1} becomes equivalent to 
\begin{equation} \label{hypotheses.2}
H_{0j}: \mu_{j}=0 \quad \text{versus} \quad H_{1j}: \mu_{j} \neq 0,\quad j=1,\ldots,p.
\end{equation}

Therefore, we shall work on the multiple testing problem under the setup (\ref{test.distribution}) and (\ref{hypotheses.2}) hereafter. Let $\calH_0 = \{j: H_{0j} \text{ is true}\}$ and $\calH_1 = \{j: H_{0j} \text{ is false}\}$ respectively be the sets of indices for the true nulls and the true alternatives, and let $p_0$ and $p_1$ denote their cardinalities. For a given threshold $t\in (0,1)$, the $j$th test is given by $t_j = \mathrm{1}(|Z_j| > |z_{t/2}|)$, where $t_j = 1$ indicates that the null hypothesis is rejected, and $t_j = 0$ otherwise; here $z_{t/2}$ denotes the $(t/2)$th quantile of a standard normal distribution. Define $V(t)=\sum_{j \in \mathcal{H}_0} t_j$, $S(t)=\sum_{j \in \mathcal{H}_1} t_j$, and $R(t)=\sum_{j =1}^p t_j$; they are respectively the number of false discoveries, correct discoveries, and the total number of discoveries. As a consequence, we have $\mathrm{FDP}(t) = V(t) / R(t)$.

\citet{fan2012estimating} studied the estimation of $\mathrm{FDP}(t)$ under an arbitrary dependence structure of $\bSigma$. One milestone finding is that if $\bSigma$ follows the weak dependence structure: 
\begin{equation} \label{weak.dependence}
\sum_{i,j}|\sigma_{ij}| = O\left(p^{2-\delta}\right) \text{ for some } \delta>0,
\end{equation} 
then,
\begin{equation} \label{asymptotic.limit}
\lim_{p\to\infty} \left[\mathrm{FDP}(t) -\frac{p_0 t}{\sum_{i=1}^{p}\left\{\Phi(z_{t/2}+\mu_{i}) + \Phi(z_{t/2}-\mu_{i})\right\}}\right] = 0,
\end{equation}
almost surely, where $\Phi(\cdot)$ denotes the cumulative distribution function of a standard normal distribution. This indicates that under the aforementioned weak dependence assumption, $\FDP(t)$ converges to a limit not depending on the covariance structure of the test statistics. It is noteworthy that this limit holds not only for $\FDP(t)$, but also for $\FDR(t)$. 

We observe that this conclusion does not hold for the asymptotic variance of $\FDP(t)$; the asymptotic variance of $\FDP(t)$ varies under different covariance structures among the test statistics even when they are weakly dependent. In this article, we shall derive the asymptotic variance of $\FDP(t)$ theoretically and study how it varies under different dependence structures of the test statistics numerically. 

\section{Asymptotic Variance of FDP}
\label{Sec-3}

In this section, we study the asymptotic variance of $\FDP(t)$ in the framework of Section \ref{Sec-2}. We first establish the asymptotic expansion of $\FDP(t)$, and its asymptotic variance under a weak dependence structure. Then, we show that with mild conditions, this asymptotic variance under dependence is higher than that under independence. This implies that under weak dependence, although the asymptotic mean of $\FDP(t)$ remains the same, its asymptotic variance can be inflated. For space limitations and presentational continuity, we have relegated the technical details in the supplementary materials. 

\subsection{Asymptotic expansion and variance of FDP}
\label{Sec-3.1}

We first derive the asymptotic expansion of the FDP in the framework of weak dependence. We introduce the following notation. For any $t\in[0,1]$, we denote $\xi_j = \Phi(z_{t/2} + \mu_j) + \Phi(z_{t/2} -\mu_j)$ and $\bar\xi = \frac{1}{p_1}\sum_{j\in \mathcal{H}_1} \xi_j$. Consider the function 
\begin{equation*}
H(\mu) = \phi(|z_{t/2}| + \mu)(|z_{t/2}| + \mu) + \phi(|z_{t/2}| - \mu)(|z_{t/2}| - \mu); 
\end{equation*}
there exists a unique root, denoted by $\mu_t$, of this function for $\mu\in (|z_{t/2}|, |z_{t/2}|+1)$. Let $C_t^{\max} = \sup_{\mu \in (-\mu_t, \mu_t)} H(\mu)$. 
  
\begin{theorem} \label{Thm-1}
Suppose $(Z_{1},\ldots,Z_{p})^{T} \sim N((\mu_{1},\ldots,\mu_{p})^{T},\bSigma)$ with unit variances, i.e., $\sigma_{jj}=1\text{ for } j=1,\ldots,p$. Assume that $(Z_{1},\dots,Z_{p})^T$ are weakly dependent as defined in (\ref{weak.dependence}), that $\limsup_{p\to \infty} p_0t/(p_1\bar \xi) < 1$, and that as $p$ is sufficiently large, for a universal constant $C>0$,
\begin{align}
&\sum_{i\neq j; i,j\in \mathcal{H}_0} \sigma_{ij}^2 \geq \frac{C_t^{\max}}{\phi(z_{t/2})|z_{t/2}|} \sum_{i\in \mathcal{H}_1, j\in \mathcal{H}_0, \mu_i\in[-\mu_t, \mu_t]} \sigma_{ij}^2, \label{Cond-1} \\
&\sum_{i\neq j; i,j\in \mathcal{H}_0} \sigma_{ij}^2 + p \geq C \sum_{i\neq j; i, j\in \mathcal{H}_1} \sigma_{ij}^2, \label{Cond-2}\\
&\sum_{i\neq j} \sigma_{ij}^4 = o\left(\sum_{i\neq j; i,j\in \mathcal{H}_0} \sigma_{ij}^2  + p_0\right). \label{Cond-3}
\end{align}
We have the following asymptotic expansion of $\FDP(t)$:
\begin{equation} \label{asymptotic.expansion}
\FDP(t) = \frac{E(\bar V)}{E(\bar R)} + m(\bar{V},\bar{R}) +r(\bar V, \bar R), 
\end{equation}
where $\bar V = V(t)/p$, $\bar R = R(t)/p$,
\begin{equation*}
m(\bar{V},\bar{R}) =  \frac{\bar{V}}{E(\bar{R})} - \frac{E(\bar{V})}{\{E(\bar{R})\}^2} \bar{R},
\end{equation*}
and the remainder term $r(\bar V, \bar R)$ satisfies that $E\{r^2(\bar V, \bar R)\} = o[\Var\{m(\bar{V},\bar{R})\}]$.
\end{theorem}

\begin{remark}
Theorem \ref{Thm-1} establishes the asymptotic expansion of $\FDP(t)$. It decomposes $\FDP(t)$ into three parts: the asymptotic mean $E(\bar{V})/E(\bar{R})$, the stochastic term $m(\bar{V},\bar{R})$ that has mean zero, and an asymptotically negligible remainder term $r(\bar V, \bar R)$. The asymptotic mean $E(\bar{V})/E(\bar{R})$ is a constant, which reinforces the conclusion given by \eqref{asymptotic.limit}; the stochastic term $m(\bar{V},\bar{R})$ is a linear combination of the tests $t_j,\ j=1\ldots, p$, which gives the asymptotic variance of $\FDP(t)$.
\end{remark}

\begin{remark}
With the weak dependence assumption (\ref{weak.dependence}),  other technical conditions in Theorem \ref{Thm-1}, i.e., Conditions \eqref{Cond-1}--\eqref{Cond-3}, are mild, if $\bSigma$ does not behave too extremely; the intuition is as follows. The number of summed terms on the left hand side of \eqref{Cond-1} is about $p_0^2$; in contrast, that on the right side is about $p_0p_1$. If $p_0 \gg p_1$, the number of summed terms on the left hand side of \eqref{Cond-1} is much more than that on the right. Therefore, this condition is satisfied unless the magnitudes of $\sigma_{ij}$ for $i\in \mathcal{H}_0,\  j\in \mathcal{H}_1$ are generally much bigger than those for $i,j\in \mathcal{H}_0$. Similar arguments are applicable for Condition \eqref{Cond-2}. For Condition \eqref{Cond-3}, the summation on the left hand side is over the fourth order of $\sigma_{ij}$, whereas that on the right is over the second order; therefore, it is easily satisfied when most of $\sigma_{ij}$ approach $0$ when $p\to \infty$.
\end{remark}

\begin{remark}
The development for Theorem \ref{Thm-1} is theoretically challenging. The expansion (\ref{asymptotic.expansion}) can be observed from a standard Taylor expansion; however, the derivation for $E\{r^2(\bar V, \bar R)\} = o[\Var\{m(\bar{V},\bar{R})\}]$ is technically involved. More specifically, two features jointly introduce challenges in our developments: (1) the leading component of $E\{r^2(\bar V, \bar R)\}$ is the summation of fourth moments of all the tests; we observe that to compare its asymptotic properties with $\Var\{m(\bar{V},\bar{R})\}$, a fourth order Taylor expansion of each moment term is further needed; (2) the dependence structure of the tests significantly complicates our study of the asymptotic upper bound of $E\{r^2(\bar V, \bar R)\}$ and the lower bound of $\Var\{m(\bar{V},\bar{R})\}$. 
\end{remark}

Following the above asymptotic expansion of $\FDP(t)$ in Theorem \ref{Thm-1}, Corollary \ref{Corrollary-asym-var} below presents its asymptotic variance.
\begin{corollary} \label{Corrollary-asym-var}
With all the conditions in Theorem \ref{Thm-1} effective, we have
\begin{equation} \label{asymptotic.variance}
\lim_{p\to\infty}\dfrac{\Var\left\{\FDP(t)\right\}}{V_1(t) + V_2(t)} = 1, 
\end{equation}
where 
\begin{align}
V_1(t) ={} & \dfrac{p_1^2 \bar{\xi}^2}{(p_0 t + p_1\bar{\xi})^4} \times p_0t(1-t) + \dfrac{p_0^2 t^2}{(p_0 t + p_1\bar{\xi})^4} \times p_1\bar{\xi}(1 - \bar{\xi}), \label{V.1}\\
V_2(t) ={} & \dfrac{2p_1^2 \bar{\xi}^2}{(p_0 t + p_1\bar{\xi})^4} \sum_{\substack{i<j \\ i,j\in\calH_0}} \Cov(t_i,t_j)  - \dfrac{2p_0p_1t\bar{\xi}}{(p_0 t + p_1\bar{\xi})^4} \sum_{\substack{i \in \calH_0 \\ j \in \calH_1}} \Cov(t_i, t_j) \notag \\
&+ \dfrac{2p_0^2 t^2}{(p_0 t + p_1\bar{\xi})^4} \sum_{\substack{i<j \\ i,j\in\calH_1}} \Cov(t_i,t_j), \label{V.2}
\end{align}
with
\begin{align}
\Cov(t_i,t_j) 
= {} & \int_{|z_i| > |z_{t/2}|} \int_{|z_j| > |z_{t/2}|} \phi(z_i, z_j; \mu_i,\mu_j, 1, 1, \sigma_{ij}) dz_i dz_j \notag \\ 
& - [\Phi(z_{t/2}+\mu_i) + \Phi(z_{t/2}-\mu_i)][\Phi(z_{t/2}+\mu_j) + \Phi(z_{t/2}-\mu_j)], \label{covariance.t}
\end{align}
in which $\phi(z_i, z_j; \mu_i,\mu_j, \sigma_i^2, \sigma_j^2, \sigma_{ij})$ is the probability density function of the bivariate normal distribution with mean $(\mu_i,\mu_j)^T$ and covariance matrix $\left(\begin{matrix} \sigma_i^2 & \sigma_{ij}\\ \sigma_{ij} & \sigma_j^2 \end{matrix}\right)$.
\end{corollary}

Corollary \ref{Corrollary-asym-var} provides an explicit formula for the asymptotic variance of FDP under weak dependence. In the special case that the test statistics are independent, $V_2(t) = 0$ and the asymptotic variance of FDP is reduced to $V_1(t)$. Therefore, the ``additional variance" of FDP due to the dependence among the test statistics is introduced by $V_2(t)$. However, the sign of $V_2(t)$ is not conclusive based on \eqref{V.2}. This leads to an interesting question: under which scenario, $V_2(t)>0$, and therefore the $\FDP$ is inflated by the variance among test statistics? To answer this question, we need to investigate the signs of the covariances between pairs of tests, i.e., $\Cov(t_i,t_j)$, which are included in the formula of $V_2(t)$. We answer this question in the subsequent subsections.

\subsection{Signs of the covariance between any pair of tests}
\label{Sec-3.2}

From our development in the subsection above, we explained that the asymptotic variance of $\FDP(t)$ depends on $V_2(t)$; and in turn depends on $\Cov(t_i,t_j)$. In this subsection, we study how the signs of $\Cov(t_i,t_j)$ are affected by the dependence among the test statistics. Interestingly, we observe that the signs of $\Cov(t_i,t_j)$ are almost certainly determined by the memberships of tests $i$ and $j$ in $\mathcal{H}_0$ and $\mathcal{H}_1$. We have the following theorem.

\begin{theorem} \label{Thm-2}
Suppose $(Z_{1},\ldots,Z_{p})^{T} \sim N((\mu_{1},\ldots,\mu_{p})^{T},\bSigma)$ with unit variances, i.e., $\sigma_{jj}=1\text{ for } j=1,\ldots,p$. Assume  $\sigma_{ij} \ne 0$ for the index pair $(i, j), i \ne j$. Then we have:
\begin{itemize}[noitemsep]
\item[$(a)$] for $i ,j\in\calH_0$, $\Cov(t_i, t_j) > 0$;
\item[$(b)$] for $i \in\calH_0$, $j \in\calH_1$, $\Cov(t_i, t_j) < 0$ if $|\mu_j| > 2|z_{t/2}|$ and $t < 2\{1-\Phi^{-1}(\sqrt{\log(3)/2})\}$;
\item[$(c)$] for $i ,j\in\calH_1$, $\Cov(t_i, t_j) > 0$ if $|\sigma_{ij}| \le |\mu|_{\min}/(|\mu|_{\max} + z_{t/2})$ and $\mathrm{sign}(\sigma_{ij}) = \mathrm{sign}(\mu_i \mu_j)$, where $|\mu|_{\min} = \min(|\mu_i|, |\mu_j|)$ and $|\mu|_{\max} = \max(|\mu_i|, |\mu_j|)$.
\end{itemize}
\end{theorem}

The motivation of the above theorem is to understand how the variance of FDP is affected by the dependence structure among the test statistics; however, these results on their own are interesting, and have theoretical and/or practical value. The theorem implies that the sign of the covariance  of any $(t_i, t_j), i\neq j$ pair is almost certainly determined by the membership of tests $i$ and $j$ in $\mathcal{H}_0$ and $\mathcal{H}_1$, and relies little on the sign of $\sigma_{ij}$; it holds for the dependence types of both weakly dependent and strongly dependent test statistics.  We expect that this theorem may benefit other multiple testing related research, when it is in need to quantify the dependence of $t_i$ and $t_j$. Next, we give some discussion on the implications of Part (a)--(c). In the discussion below, we assume that $Z_i$ and $Z_j$ are dependent, if not stated otherwise. 

Part (a) states that $t_i$ and $t_j$ are positively correlated provided that both corresponding tests belong to $\mathcal{H}_0$. The intuition is that both tests are two-sided tests and the corresponding test statistics have zero expectations. Therefore, $|Z_i|$ and $|Z_j|$ are positively correlated, which further leads to the positive correlation between $t_i$ and $t_j$. 

Part (b) indicates that when one test, say $i$, belongs to $\mathcal{H}_0$, and the other test, say $j$, belongs to $\mathcal{H}_1$, then $t_i$ and $t_j$ are negatively correlated given that $t$ is a reasonable threshold, i.e., $t < 2\{1-\Phi^{-1}(\sqrt{\log(3)/2})\}\approx 0.46$ and that the signal of $j$ is not too weak, i.e., $|\mu_j| > 2|z_{t/2}|$. Note that both conditions are typically satisfied in practical problems. The intuition of this conclusion is as follows. Let $Z_j^* = Z_j - \mu_j$; then if we consider $Z_j^*$ as the test statistic, its corresponding test belongs to $\mathcal{H}_0$; its correlations with the other test statistics remain unchanged. The corresponding test $t_j$ in the notation of $Z_j^*$ has the form: 
\begin{eqnarray*}
t_j = \mathrm{1}(|Z_j| > |z_{t/2}|) = \mathrm{1}\{Z^*_j \in (-\infty, -|z_{t/2}| - \mu_j) \cup (|z_{t/2}| - \mu_j, \infty)\},
\end{eqnarray*}
where we assume that $\mu_j>0$ for presentational convenience; 
this leads to $$1-t_j = \mathrm{1}\left\{Z^*_j \in (-|z_{t/2}| - \mu_j, |z_{t/2}| - \mu_j)\right\},$$ and implies that the acceptance region for the test based on $Z_j^*$ is contained by the rejection region of $t_i$ given that the signal $\mu_j > 2|z_{t/2}|$. From our conclusion in Part (a), $1-t_j$ is positively correlated with $t_i$, and therefore $t_j$ and $t_i$ are negatively correlated. 

Part (c) gives sufficient conditions under which $t_i$ and $t_j$ are positively correlated when $i,j\in \mathcal{H}_1$. This is the only case where $\sigma_{ij}$ affects the sign of the correlation between the corresponding tests. These conditions are also intuitively met in real applications. In particular, the first condition $|\sigma_{ij}| \le |\mu|_{\min}/(|\mu|_{\max} + z_{t/2})$ essentially requires that the magnitude of $\sigma_{ij}$ does not exceed a ratio of the minimum to the maximum signals. One special practical scenario is that if the two signals are reasonably strong and in the similar scale, this ratio is close to 1; thus this condition is met. The second condition $\mathrm{sign}(\sigma_{ij}) = \mathrm{sign}(\mu_i \mu_j)$ needs that the sign of $\sigma_{ij}$ coincides with the signs of the corresponding signals: the test statistics are positively correlated if the signals are of the same signs; and they are negatively correlated otherwise. More specifically, consider the GWAS example, Part (c) implies that the tests for the significance of the $i$th and $j$th SNPs are positively correlated if their individual effects on the disease status are both positive or negative.


\subsection{Inflation of the FDP variance under dependence}
\label{Sec-3.3}

In Section \ref{Sec-3.1}, we discussed that $V_2(t)$ attributes to the FDP variance difference among varied dependence structures of the test statistics. The signs of its components are further discussed in Section \ref{Sec-3.2}. From these discussion, we conclude that the variance of FDP is inflated given that Parts (a)--(c) in Theorem \ref{Thm-2} all hold. We summarize this result in the following theorem.

\begin{theorem} \label{Thm-3}
Suppose $(Z_{1},\ldots,Z_{p})^{T} \sim N((\mu_{1},\ldots,\mu_{p})^{T},\bSigma)$ with $\sigma_{jj}=1\text{ for } j=1,\ldots,p$. For any $i,j\in \mathcal{H}_1, i\neq j$, assume that $|\mu_i| \ge 2|z_{t/2}| > 2\sqrt{\log(3)/2}$,  $|\sigma_{ij}| \le |\mu|_{\min}/(|\mu|_{\max} + z_{t/2})$, and $\mathrm{sign}(\sigma_{ij}) = \mathrm{sign}(\mu_i \mu_j)$, where $|\mu|_{\min} = \min(|\mu_i|, |\mu_j|)$ and $|\mu|_{\max} = \max(|\mu_i|, |\mu_j|)$. Then, $V_2(t)>0$.
\end{theorem}

Theorem \ref{Thm-3} implies that the variance of FDP is inflated if the test statistics satisfy the given assumptions on their dependence structure. On the one hand, as we have discussed in Section \ref{Sec-3.2}, these are mild conditions for practical data. On the other hand, they are only one set of sufficient conditions that lead to the inflation of the FDP variance. 
In our numerical studies, we experiment with some settings that violate the conditions in Theorem \ref{Thm-3}, and still observe the inflation of the variance of FDP; please refer to Section \ref{Sec-6} for details. Such scenarios are not theoretically justified, but may often be practically observed.

\section{Estimation of the Asymptotic Variance of FDP}
\label{Sec-5}

In Section \ref{Sec-3.1}, we have established the asymptotic variance of FDP in the framework of weak dependence. We observe that this variance depends on the parameters $(\mu_1, \ldots, \mu_p)^T$ and $\bSigma$. We need to estimate them so that we can quantify this variance to make it practically applicable. Throughout, we assume that $\bSigma$ is known or can be estimated from other resources (see Section \ref{Sec-2}) and estimate only $(\mu_1,\ldots, \mu_p)^T$. 

One direct approach is to estimate $(\mu_1, \ldots, \mu_p)^T$ by their corresponding test statistics $(Z_1, \ldots, Z_p)^T$. However, in practical data, we often have $p_0 \gg p_1$; this implies that most $\mu_j$'s are zeros. Estimating these ``zeros" using the corresponding observed $Z_j$'s tends to diminish the efficiency of the method, as they can be viewed as noise. In our numerical studies, we consider the following estimation approach.

\begin{itemize}
\item[(1)] Find an estimate $\hat \pi_0$ for $\pi_0 = p_0/p$, which is the proportion of the number of tests that belong to $\mathcal{H}_0$ over the total number of tests. 
    
Estimation methods for $\pi_0$ are available in the literature. For example, \citet{storey2003statistical} proposed $\hat\pi_0(\lambda) = \sum_{i=1}^p 1(P_i > \lambda)/\{m (1 - \lambda)\}$ with an $\lambda\in (0,1)$; the basic idea of this estimator is to pretend that $p$-values greater than $\lambda$ are from $\mathcal{H}_0$. In this framework, smoothed and bootstrap estimators of $\pi_0$ are proposed by \citet{storey2003statistical} and \citet{storey2004strong}, respectively. Additionally, \citet{langaas2005estimating} proposed a nonparametric maximum likelihood estimator of $\pi_0$. More recently, \citet{wang2010slim} proposed a ``SLIM" estimator of $\pi_0$; their estimator accounted for the dependence among the test statistics. We shall adopt these estimators in our numerical studies and compare their performance in the estimation of the FDP variance. 
    
\item[(2)] Estimate $p_1$ by $\hat p_1 = p(1-\hat \pi_0)$. We then collect the hypotheses with the top $\hat p_1$ statistic values to form $\widehat{\mathcal{H}}_1$, and the rest are assigned to $\widehat{\mathcal{H}}_0$. 

\item[(3)] Set $\hat \mu_j = 0$ for $j\in \widehat{\mathcal{H}}_0$; and set $\hat \mu_j = Z_j$ for $j\in \widehat{\mathcal{H}}_1$; and use them to estimate the variance of $\FDP(t)$ based on Corollary \ref{Corrollary-asym-var}.  
\end{itemize}

\section{Simulation Studies}
\label{Sec-6}

We conduct simulation studies to validate our developments in Sections \ref{Sec-3} and \ref{Sec-5}. In Section \ref{Sec-6.1}, we use various numerical examples to illustrate how the asymptotic variances of FDP are different for varied dependence structures among the test statistics. In Section \ref{Sec-6.3}, we illustrate how the choice of the methods for estimating $\hat \pi_0$ reviewed in Section \ref{Sec-5} affects the estimation for the variance of FDP. 

\subsection{Variance of FDP under varied dependence structures}
\label{Sec-6.1}

Throughout, we generate $(Z_1,\ldots,Z_p)^T \sim N((\mu_1,\ldots,\mu_p)^T, \bSigma)$ with $p = 2000$. We examined three choices of $p_1$: 50, 100, and 200, where the indices of tests in $\mathcal{H}_1$ are randomly sampled from $\{1,\ldots,p\}$; for $j\in \calH_1$, we set $\mu_j = 2z_{t/2}$. We considered three possible values of $t$:  0.005, 0.02, and 0.05. The choice of $\bSigma$ determines the dependence structure among the test statistics. We use similar methods as \citet{fan2012estimating} to generate it; the details are given below. 
\begin{itemize}
\item[(1)] We generate a random sample $\bX_1, \ldots, \bX_{400}$, which are iid copies of $\bX = (X_1, \ldots, X_p)^T$. We consider seven models for the joint distribution of $\bX$.

\begin{itemize}
\item \textbf{[M1]} $X_1,\ldots, X_p$ are generated as iid $N(0,1)$ random variables. 
\item \textbf{[M2]} Let $\bX \sim N_{p}(0, \mathbf\Lambda)$, where $\mathbf\Lambda$ has diagonal elements $1$ and off-diagonal elements $1/2$.
\item \textbf{[M3]} Let $\{X_{k}\}_{k=1}^{1900}$ be iid $N(0,1)$ random variables, and $X_{k} = \sum_{l=1}^{10}X_{l}(-1)^{l+1}/5 + \sqrt{1-\dfrac{10}{25}}\varepsilon_{k}$ for $k=1901,\dots,2000,$
with $\{\varepsilon_{k}\}_{k=1901}^{2000}$ being $N(0,1)$ random variables.
\item \textbf{[M4]} Let $\{X_{k}\}_{k=1}^{2000}$ be iid Cauchy random variables with location parameter $0$ and scale parameter $1$.
\item \textbf{[M5]} Let $X_{j} = \rho_{j}^{(1)}W^{(1)}+\rho_{j}^{(2)}W^{(2)}+\rho_{j}^{(3)}W^{(3)}+H_{j}$, where $W^{(1)} \sim N(-2,1)$, $W^{(2)} \sim N(1,1)$, $W^{(3)} \sim N(4,1)$, $\rho_{j}^{(1)}$, $\rho_{j}^{(2)}$, $\rho_{j}^{(3)} \sim U(-1,1)$, $H_{j}\sim N(0,1)$; and all these random variables are independent.
\item \textbf{[M6]} Let $X_{j} = \rho_{j}^{(1)}W^{(1)}+\rho_{j}^{(2)}W^{(2)}+H_{j}$, where $W^{(1)}, W^{(2)}\sim N(0,1)$, $\rho_{j}^{(1)}, \rho_{j}^{(2)} \sim U(-1,1)$, $H_{j} \sim N(0,1)$, and all these random variables are independent.
\item \textbf{[M7]} Let $X_{j} = \mathrm{sin}\left(\rho_{j}^{(1)}W^{(1)}\right)+\mathrm{sign}\left(\rho_{j}^{(2)}\right)\exp\left(|\rho_{j}^{(2)}|W^{(2)}\right)+H_{j}$, where $W^{(1)}, W^{(2)}\sim N(0,1)$, $\rho_{j}^{(1)}, \rho_{j}^{(2)}\sim U(-1,1)$, and $H_{j}\sim N(0,1)$, and all these random variables are independent.
    \end{itemize}
    
\item[(2)]  We compute $\bSigma^{\text{initial}}$ as the sample correlation matrix based on $\bX_1, \ldots, \bX_{400}$. Note that for some scenarios above, the components of $\bX$ are strongly dependent
. 

\item[(3)]  We adopt the PFA method in \citet{fan2012estimating} to generate a weakly dependent $\bSigma$ based on $\bSigma^{\text{initial}}$. When applying the PFA method, we used the criterion $p^{-2}\sum_{1\le i,j \le p}|\sigma_{ij}| < 0.05$ numerically. 

\end{itemize}

With these setups, for \textbf{M1}, the test statistics are independent, whereas for the other models, the test statistics are weakly dependent with varied dependence structures. For each aforementioned $p_1$, $t$, and model combination, we conduct $\FDP$ analysis, and evaluate the corresponding sample standard deviation of $\FDP$s over 1000 repetitions (named ``Emp" in Table \ref{Table-2}); we also compute the corresponding asymptotic standard deviation of $\FDP$ given in Corollary \ref{Corrollary-asym-var} with $\bmu$ and $\bSigma$ being replaced with their simulated values (named ``Asym" in Table \ref{Table-2}). The results are summarized in Table \ref{Table-2}. 

\begin{table} 
\centering
 \caption{Comparison for the standard deviation ($\times 100$) of FDPs}
\begin{tabular}{cccccccccc}
\hline\hline
$t$ & $p_1$ & Method & \textbf{M1} & \textbf{M2} & \textbf{M3} & \textbf{M4} & \textbf{M5} & \textbf{M6} & \textbf{M7}\\ \hline
& 50 & Asym & 4.37 & 6.11 & 5.13 & 6.91 & 6.11 & 6.12 & 6.15 \\
 && Emp & 4.43  & 6.08 & 5.16 & 6.60 & 5.73 &6.20 & 5.90\\
0.005 & 100 & Asym & 2.57 & 3.57 & 3.01 & 4.03 & 3.57 & 3.57 & 3.59 \\
 && Emp & 2.65 & 3.56 & 2.94 & 3.85 & 3.68 & 3.66 & 3.52 \\
& 200 & Asym & 1.37 & 1.88 & 1.59 & 2.12 & 1.88 & 1.89 & 1.89 \\
 && Emp & 1.33 & 1.83 & 1.60 & 2.06 & 1.86 & 1.79 & 1.95 \\ \hline
& 50 & Asym & 3.92 & 6.74 & 5.23 & 7.12 & 6.75 & 6.75 & 6.80 \\
 && Emp & 3.90 & 6.64 & 5.35 & 7.30 & 6.91 & 6.90 & 6.84\\
0.02 & 100 & Asym & 3.23 & 5.50 & 4.28 & 5.81 & 5.51 & 5.51 & 5.55 \\
 && Emp & 3.20 & 5.50 & 4.44 & 5.66 & 5.39 & 5.58 & 5.35\\
& 200 & Asym & 2.15 & 3.61 & 2.82 & 3.80 & 3.61 & 3.61 & 3.64 \\
 && Emp & 2.14 & 3.61 & 2.84 & 3.70 & 3.44 & 3.51 & 3.40\\ \hline
& 50 & Asym & 2.25 & 4.33 & 3.25 & 4.42 & 4.34 & 4.34 & 4.37 \\
 && Emp & 2.27 & 4.54 & 3.25 & 4.47 & 4.58 & 4.45 & 4.54\\
0.05 & 100 & Asym & 2.53 & 4.87 & 3.65 & 4.97 & 4.88 & 4.88 & 4.92 \\
 && Emp & 2.54 & 4.70 & 3.60 & 5.04 & 4.80 & 4.92 & 5.02\\
& 200 & Asym & 2.23 & 4.24 & 3.19 & 4.32 & 4.24 & 4.24 & 4.27 \\
 && Emp & 2.25 & 4.25 & 3.22 & 4.30 & 4.09 & 4.17 & 4.22\\ \hline
 \end{tabular} \label{Table-2}
 \end{table}

From this table, we observe that for all cases, the asymptotic standard deviation is very close to the sample standard deviation of $\FDP$s over 1000 repetitions. This validates our conclusion in Corollary \ref{Corrollary-asym-var}, and indicates that this asymptotic variance well approximates its corresponding true value in practice. Furthermore, for all the $p_1$ and $t$ combinations, the standard deviations of the FDP from \textbf{M1} are significantly smaller than those from the other models. This implies that the variances of FDP are inflated when dependence is present among the test statistics. Note that for these examples, the conditions in Theorem \ref{Thm-3} are not necessarily guaranteed; this reinforces our discussion in Section \ref{Sec-3.3} and demonstrates that the inflation of the FDP variance may occur in wide real applications. Among \textbf{M2}--\textbf{M6}, we observe that \textbf{M3} leads to the smallest inflation; this could be because that this model is close to the independent case, as the majority of $X_k$'s are independent, i.e., for $k=1,\ldots,1900$.

\subsection{Estimation of the asymptotic variance for FDP}
\label{Sec-6.3}

In this subsection, we assess the performance of the method in Section \ref{Sec-5} in estimating the asymptotic variance of FDP. As we observe in Section \ref{Sec-5}, our method relies on $\pi_0$, whose estimation methods are available in the literature. We incorporate these methods into our method and compare their performance. In particular, we consider: (1) the smoothed estimate by \citet{storey2003statistical}, named as ``Smoothed"; (2) the bootstrap estimate by \citet{storey2004strong}, named as ``Bootstrap"; (3) the nonparametric maximum likelihood estimate by \citet{langaas2005estimating}, named as ``Langaas"; and (4) the SLIM estimate by \citet{wang2010slim}, named as ``SLIM". We include the asymptotic standard deviation of FDP based on the true values of the parameters, i.e, $\bmu$ and $\bSigma$, for each case as a reference, named as ``True".  

We adopt the methods in Section \ref{Sec-6.1} to generate test statistics, and apply the aforementioned methods to estimate the asymptotic variances of FDPs. We only present the results for Model \textbf{M4} in Section \ref{Sec-6.1}; the results for the other models show similar patterns, and are omitted due to space limitation. The means and standard deviations of the estimated asymptotic standard deviations of FDPs over 1000 replicates are displayed in Table \ref{Table-3}. From this table, we observe that the Langaas and SLIM methods perform reasonably well in the estimation of the asymptotic standard deviation of FDPs; in particular, these methods result in estimates close to the true asymptotic standard deviation based on our Corollary \ref{Corrollary-asym-var} and lead to smaller standard deviations. We observe that this is because in the estimation of $\pi_0$, these methods are more robust to the dependence among the test statistics, whereas the Smoothed and Bootstrap methods can work well only when the test statistics are independent. We provide the average estimated number of identified alternative hypotheses, i.e., $p_1 = p (1-\pi_0)$ and the corresponding standard deviations over 1000 replicates in Table \ref{Table-4}; from this table, we observe that the performance in the estimation of $p_1$ is in the order of 
\begin{eqnarray*}
\mbox{Bootstrap} < \mbox{Smoothed} < \mbox{Langaas} < \mbox{SLIM}, 
\end{eqnarray*}
which is consistent with our observation in Table \ref{Table-3}. 




\begin{table} 
\centering
\caption{Mean ($\times 1000$) and standard deviation ($\times 1000$, in parenthesis) of the estimated asymptotic standard deviation of FDPs over 1000 replicates.}
\begin{tabular}{ccccccc}
\hline\hline
$t$ & $p_1$ & True & Smoothed & Bootstrap & Langaas & SLIM \\ \hline
& 50 & 6.91 & 3.46 (3.12) & 5.83 (2.70) & 6.00 (1.02) & 7.31 (1.52) \\
0.005 & 100 & 4.03 & 3.13 (2.82) & 4.01 (1.43) & 3.68 (0.42) & 4.23 (0.60) \\
& 200 & 2.12 & 2.65 (2.45) & 2.06 (0.37) & 2.02 (0.14) & 2.18 (0.16) \\ \hline
& 50 & 7.12 & 3.57 (2.92) & 5.81 (1.79) & 6.70 (0.62) & 6.97 (0.43) \\
0.02 & 100 & 5.81 & 3.56 (2.48) & 5.39 (1.02) & 5.45 (0.54) & 5.90 (0.40) \\
& 200 & 3.80 & 3.46 (1.84) & 3.63 (0.57) & 3.64 (0.29) & 3.92 (0.22) \\ \hline
& 50 & 4.42 & 2.87 (2.35) & 4.24 (1.50) & 4.74 (0.40) & 4.32 (0.69) \\
0.05 & 100 & 4.97 & 3.04 (2.13) & 4.71 (0.49) & 4.91 (0.16) & 4.95 (0.09) \\
& 200 & 4.32 & 3.51 (1.52) & 4.16 (0.37) & 4.23 (0.27) & 4.41 (0.12) \\ \hline
\end{tabular} \label{Table-3}
\end{table}
 
\begin{table} 
\centering
\caption{Mean and standard deviation (in parenthesis) of the estimated number of alternative hypotheses over 1000 replicates.}
\begin{tabular}{cccccc}
\hline\hline
$t$ & $p_1$ & Smoothed & Bootstrap & Langaas & SLIM \\ \hline
& 50 & 103 (122) & 107 (80) & 90 (53) & 55 (20) \\
0.005 & 100 & 136 (135) & 153 (86) & 138 (51) & 102 (21) \\
& 200 & 214 (152) & 251 (77) & 239 (49) & 202 (20) \\ \hline
& 50 & 103 (122) & 110 (89) & 90 (55) & 55 (21) \\
0.02 & 100 & 133 (136) & 152 (85) & 137 (51) & 103 (21) \\
& 200 & 205 (154) & 251 (80) & 237 (50) & 202 (20) \\ \hline
& 50 & 94 (117) & 108 (85) & 90 (58) & 53 (21)  \\
0.05 & 100 & 129 (135) & 150 (84) & 136 (54) & 100 (20) \\
& 200 & 207 (147) & 252 (80) & 234 (56) & 197 (20) \\ \hline
\end{tabular} \label{Table-4}
\end{table}

\section{Real Data Analysis}

We apply our method to a genome-wide association study (GWAS); more specifically, an eQTL mapping study. The data are formed of 210 subjects from the International Hapmap Project, including three groups: (1) 90 subjects from Asian formed of 45 Japanese from Tokyo, 45 Han Chinese from Beijing, named ``Grp1"; (2) 60 Utah residents with ancestry from northern and western Europe, named ``Grp2"; and (3) 60 Nigerian from Ibadan, named ``Grp3". More details of the data can be found from the web link: \url{ftp://ftp.sanger.ac.uk/pub/genevar/}; see also \citet{bradic2011penalized}.  We adopt the existing methods to construct test statistics for testing the association between the expression level of the gene  CCT8 and the SNP genotypes for each group of the subjects \citep{bradic2011penalized, fan2012estimating}.

The SNP genotypes have three possibilities:  $\text{SNP} = 0$ means no polymorphism, $\text{SNP} = 1$ indicates one nucleotide has polymorphism, and $\text{SNP} = 2$ implies both nucleotides have polymorphisms. Adopting the strategy in \citet{bradic2011penalized}, we transform the SNP data into two dummy variables, namely $(d_1, d_2)$:
\begin{itemize}
\item $(d_1, d_2) = (0, 0)$ for $\text{SNP} = 0$;
\item $(d_1, d_2) = (1, 0)$ for $\text{SNP} = 1$;
\item $(d_1, d_2) = (0, 1)$ for $\text{SNP} = 2$.
\end{itemize}
For subject $i$, denote by $Y_i$ and $(d_{1,i,j}, d_{2,i,j})$ the logarithm-transformed value of the expression level of gene CCT8 and the aforementioned dummy variables of the $j$th SNP. We consider two sets of marginal linear regression models: 
\begin{eqnarray*}
\mbox{Set 1:} \qquad  Y_i &=& \alpha_{1,j} + \beta_{1,j} d_{1,i,j} + \epsilon_{1,i,j}, \quad \mbox{for } j = 1,\ldots, p; \\
\mbox{Set 2:} \qquad Y_i &=& \alpha_{2,j} + \beta_{2,j} d_{2,i,j} + \epsilon_{2,i,j}, \quad \mbox{for } j = 1,\ldots, p. 
\end{eqnarray*}
Here we impute the missing SNP data by 0, and exclude the SNPs that are identical to each other for all the subjects. 

For each group of the data, we combine aformentioned two sets of models, and follow the procedure in \citet{fan2012estimating} to generate the test statistics; this leads to $2p$ test statistics. In detail, we use the procedure described in Section \ref{Sec-2} to establish an initial set of test statistics and their corresponding covariance matrix. We then use the dependence-adjusted procedure in \citet{fan2012estimating} to remove the top ten principal factors from the covariance matrix, and take a standardization step to generate a new set of test statistics of unit marginal variances; this helps reduce the dependence among the test statistics.

With the resultant test statistics for each group, we estimate the asymptotic limits and standard deviations of $\FDP(t)$ based on the following setups:
\begin{itemize}

    \item For each group of test statistics, we consider two thresholds of $p$-value, which are chosen such that the estimates of the $\FDP(t)$ limits are reasonable values. 
    
    \item The asymptotic limits are estimated from \eqref{asymptotic.limit}, denoted by $\widehat{\text{FDP}}(t)$.
    
    \item The asymptotic standard deviations are computed based on the methods in Section \ref{Sec-5}. As we have observed in Section \ref{Sec-6.3} that the standard deviation estimates based on the $\pi_0$ estimates from the ``Langaas" and ``SLIM" methods outperform the others; therefore, we present only the results from these methods. Furthermore, to assess how dependence may affect the estimation of these standard deviation estimates, we report these estimates based on the assumptions that the test statistics are independent and dependent, respectively. Therefore, for each group of subjects, we have four estimates, respectively denoted by $S_{\widehat{\text{FDP}}(t)}^{\text{Langaas}}$, $S_{\widehat{\text{FDP}}(t), \text{Ind}}^{\text{Langaas}}$, $S_{\widehat{\text{FDP}}(t)}^{\text{SLIM}}$,  and $S_{\widehat{\text{FDP}}(t), \text{Ind}}^{\text{SLIM}}$. 
\end{itemize}

We summarize our estimation results in Table \ref{Table-5}. From this table, we observe that under different assumptions of whether the dependence is present among the test statistics 
significantly affect the asymptotic standard deviation of $\FDP(t)$. The ratios  $S_{\widehat{\text{FDP}}(t)}^{\text{Langaas}}/S_{\widehat{\text{FDP}}(t), \text{Ind}}^{\text{Langaas}}$ and $S_{\widehat{\text{FDP}}(t)}^{\text{SLIM}}/S_{\widehat{\text{FDP}}(t), \text{Ind}}^{\text{SLIM}}$ range from 4 to 7 folds. Moreover, the ratios of the FDP standard deviation estimates to its limit values $\widehat{\text{FDP}}(t)$ are between 0.25 and 0.82 and thus the uncertainty of FDP is nonnegligible. These observations reinforce the statement that in an FDP related inference, it may not be sufficient to consider only the limit value of FDP (or equivalently, FDR); we suggest that it is important to correctly account for the dependence; otherwise, it might be over optimistic in the conclusion of how well the FDP is controlled. For example, when $t = 0.005$ in Grp1, the 95 percentile (say, two standard deviations above the mean) of the $\FDP(t)$ is estimated as $0.315 + 0.153 \times 2 = 0.621$ under dependence; in contrast, if we assume that the test statistics are independent, the corresponding estimate is $0.315 + 0.031 \times 2= 0.377$. Therefore, ignoring the dependence leads to an over optimistic belief of how well the upper bound of FDP is controlled; such a nonignorable difference may significantly change the statistical interpretation of the results in practice.

\begin{table} 
\centering
\caption{Estimation results for the GWAS study.} 
\begin{tabular}{cccccccccc}
\hline\hline
Population & \#SNP & $t$ & $\widehat{\text{FDP}}(t)$ & $S_{\widehat{\text{FDP}}(t)}^{\text{Langaas}}$ & $S_{\widehat{\text{FDP}}(t), \text{Ind}}^{\text{Langaas}}$ & $S_{\widehat{\text{FDP}}(t)}^{\text{SLIM}}$ & $S_{\widehat{\text{FDP}}(t),\text{Ind}}^{\text{SLIM}}$\\
\hline
Grp1 & 73 & 0.005 & 0.315 & 0.153 & 0.031 & 0.259 & 0.052\\
 & 105 & 0.01 & 0.377 & 0.166 & 0.031 & 0.244 & 0.045\\ \hline
Grp2 & 166 & 0.01 & 0.199 & 0.107 & 0.018 & 0.148 & 0.025\\
 & 199 & 0.02 & 0.441 & 0.125 & 0.019 & 0.167 & 0.025\\ \hline
Grp3 & 89 & 0.002 & 0.147 & 0.069 & 0.017 & 0.092 & 0.022\\
 & 135 & 0.005 & 0.371 & 0.092 & 0.019 & 0.124 & 0.025\\
 \hline
\end{tabular} \label{Table-5}
\end{table}

\section{Discussion}

In this article, we have established the theoretical results for the asymptotic variance of $\FDP$ in the framework that the test statistics are weakly dependent. We observe that with mild conditions, the asymptotic variance of $\FDP$ under dependence is usually greater than that under independence; this is also demonstrated by our numerical studies: even under weak dependence, the variances of the FDPs can be numerically much larger than those under independence. It is thus crucial for us to take into account the dependence in the estimation of the variance of the FDP and report it together with its mean. Furthermore, we have observed in the real data example that the standard deviation estimate of $\FDP$ under the dependence assumption is significantly different from (usually greater than) that under the independence assumption. In practice, we suggest that we presumably include the dependence in the estimation, unless there is a strong evidence of otherwise. This will be helpful for us to better understand how well the FDP is controlled, and avoid over optimistic conclusions.   

In this article, we have studied the asymptotic uncertainty of FDP in the same framework as \citet{fan2012estimating}. Under the assumption that the test statistics are weakly dependent, we obtain the asymptotic expansion of $\FDP(t)$ and establish the explicit formula for computing the asymptotic variance of $\FDP(t)$. Within this framework, we observe many interesting open problems that can be studied in the future. We list several as examples; there are many others. 

First, our proposed method for estimating the variance of $\FDP(t)$ in Section \ref{Sec-5} relies on the estimation of $\pi_0$; there, we adopted existing methods in the literature. From simulation studies in Section \ref{Sec-6.3}, we observe that the performance of the variance estimate for $\FDP(t)$ relies heavily on that of $\pi_0$; although the Langaas and SLIM methods perform reasonably well, we believe that they are not yet optimal, as none of these $\pi_0$ estimation methods have sufficiently accounted for the dependence among the test statistics. Therefore, we expect that establishing improved $\pi_0$ estimation methods that can better accommodate the dependence among the test statistics would be of particular interest, and can be incorporated into our estimation method in Section \ref{Sec-5}. 

Second, we have imposed some weak dependence assumptions in our development; this may limit the method in some applications. Practically, one possible solution is not to use the original test statistics, but
the dependence-adjusted ones, as they are more powerful and often weakly dependent \citep{friguet2011estimation, fan2012estimating}; yet there is no theoretical guarantees. Therefore, it is challenging, but would be of interest to study the asymptotic variance of $\FDP(t)$ under arbitrary dependence. 

Third, we assume that the test statistics are normally distributed in this article. It would be of great both practical and theoretical value to generalize our study to accommodate multiple testing problems where the test statistics follow other popular distributions, e.g., $t$, $F$, and chi-square distributions. We refer to \citet{zhuo2020test} as an example of the recent works on dependent $t$-tests and leave the relevant investigation to our future work.

\begin{center}
{SUPPLEMENTARY MATERIALS}
\end{center}
The supplementary materials contain
technical details for the theorems in Section \ref{Sec-3}.

\bibliographystyle{asa}
\bibliography{FDP}

\end{document}


\def\spacingset#1{\renewcommand{\baselinestretch}%
{#1}\small\normalsize} \spacingset{1}


\if1\blind
{
  \title{\bf Asymptotic Uncertainty of False Discovery Proportion: Supplementary Materials}
  \author{Meng Mei$^\text{a}$, Tao Yu$^\text{b}$, and Yuan Jiang$^\text{a}$\thanks{Yuan Jiang is the corresponding author. Meng Mei and Yuan Jiang's research is supported in part by National Institutes of Health grant R01 GM126549. Tao Yu's research is supported in part by the Singapore Ministry Education Academic Research Tier 1 Funds: R-155-000-202-114.}}
  \affil{$^\text{a}$ Department of Statistics, Oregon State University\\
         $^\text{b}$ Department of Statistics \& Data Science, National University of Singapore}
  \date{}
  \maketitle
} \fi

\if0\blind
{
  \bigskip
  \bigskip
  \bigskip
  \begin{center}
    {\LARGE\bf Asymptotic Uncertainty of False Discovery Proportion: Supplementary Materials}
\end{center}
  \medskip
} \fi





\spacingset{1.9} 

\section{Notation and Review of Technical Results in the Main Article} \label{sec-notation-review}

We need the following notation in our development. Throughout, we use ``$\lesssim$" (``$\gtrsim$") to denote smaller (greater) than up to a universal constant. For any $t\in[0,1]$, we denote $\xi_j = \Phi(z_{t/2} + \mu_j) + \Phi(z_{t/2} -\mu_j)$ and $\bar\xi = \frac{1}{p_1}\sum_{j\in \mathcal{H}_1} \xi_j$. Consider the function 
\begin{equation*}
H(\mu) = \phi(|z_{t/2}| + \mu)(|z_{t/2}| + \mu) + \phi(|z_{t/2}| - \mu)(|z_{t/2}| - \mu); 
\end{equation*}
there exists a unique root, denoted by $\mu_t$, of this function for $\mu\in (|z_{t/2}|, |z_{t/2}|+1)$. Let $C_t^{\max} = \sup_{\mu \in (-\mu_t, \mu_t)} H(\mu)$. 

Furthermore, recall that in the main article, the $i$th test is given by $t_i = \mathrm{1}(|Z_i| > |z_{t/2}|) = \mathrm{1}(P_i < t)$.
Let 
\begin{eqnarray*}f_{\{t_1,  \dots, t_k\}}^{(i_1,  \dots, i_k)}(a, t) = \dfrac{\partial^{\sum_{j = 1}^{k}i_j} f_{\{t_1, \ldots, t_k\}}( \rho_1, \ldots, \rho_k; a, t)}{\partial (\rho_1)^{i_1}   \ldots \partial(\rho_k)^{i_k}}\Big|_{\rho_j = 0, j=1,\ldots, k},
\end{eqnarray*}
where 
$$ f_{\{t_1,  \ldots, t_k\}}( \rho_1, \ldots, \rho_k; a, t) = \Phi \left(\dfrac{z_{t/2} - \mu_a - \rho_1x_{t_1} -  \ldots - \rho_kx_{t_k}}{\sqrt{1 - \rho_1^2  - \ldots -\rho_k^2}}\right).$$
Let $\xi_i = E(t_i)$ for $i\in \mathcal{H}_1$; clearly we have $\xi_i = \Phi(z_{t/2} + \mu_i) + \Phi(z_{t/2} - \mu_i)$. Denote $\bar \xi = \frac{1}{p_1} \sum_{i=1}^{p_1} \xi_i$. Then the converged limit of $\mbox{FDP}(t)$ on the right hand side of (6) in the main article can be written to be $\frac{p_0t}{p_0t + p_1 \bar \xi}$.

The following theorem is Theorem 1 in Section 3.1 of the main article; it establishes the asymptotic expansion of the $\FDP(t)$ in the framework of the weak dependence. 

\begin{theorem} \label{Thm-1}
Suppose $(Z_{1},\ldots,Z_{p})^{T} \sim N((\mu_{1},\ldots,\mu_{p})^{T},\bSigma)$ with unit variances, i.e., $\sigma_{jj}=1\text{ for } j=1,\ldots,p$. Assume that $(Z_{1},\dots,Z_{p})^T$ are weakly dependent as defined in (5) in the main article, that $\limsup_{p\to \infty} p_0t/(p_1\bar \xi) < 1$, and that as $p$ is sufficiently large, for a universal constant $C>0$,
\begin{align}
&\sum_{i\neq j; i,j\in \mathcal{H}_0} \sigma_{ij}^2 \geq \frac{C_t^{\max}}{\phi(z_{t/2})|z_{t/2}|} \sum_{i\in \mathcal{H}_1, j\in \mathcal{H}_0, \mu_i\in[-\mu_t, \mu_t]} \sigma_{ij}^2, \label{Cond-1} \\
&\sum_{i\neq j; i,j\in \mathcal{H}_0} \sigma_{ij}^2 + p \geq C \sum_{i\neq j; i, j\in \mathcal{H}_1} \sigma_{ij}^2, \label{Cond-2}\\
&\sum_{i\neq j} \sigma_{ij}^4 = o\left(\sum_{i\neq j; i,j\in \mathcal{H}_0} \sigma_{ij}^2  + p_0\right). \label{Cond-3}
\end{align}
We have the following asymptotic expansion of $\FDP(t)$:
\begin{equation} \label{asymptotic.expansion}
\FDP(t) = \frac{E(\bar V)}{E(\bar R)} + m(\bar{V},\bar{R}) +r(\bar V, \bar R), 
\end{equation}
where $\bar V = V(t)/p$, $\bar R = R(t)/p$,
\begin{equation*}
m(\bar{V},\bar{R}) =  \frac{\bar{V}}{E(\bar{R})} - \frac{E(\bar{V})}{\{E(\bar{R})\}^2} \bar{R},
\end{equation*}
and the remainder term $r(\bar V, \bar R)$ satisfies that $E\{r^2(\bar V, \bar R)\} = o[\Var\{m(\bar{V},\bar{R})\}]$.
\end{theorem}

A direct outcome of this theorem is the following Corollary, which establishes the explicit formula for the asymptotic variance of $\FDP(t)$. It is presented as Corollary 1 in the main article. 

\begin{corollary} \label{Corrollary-asym-var}
With all the conditions in Theorem \ref{Thm-1} effective, we have
\begin{equation} \label{asymptotic.variance}
\lim_{p\to\infty}\dfrac{\Var\left\{\FDP(t)\right\}}{V_1(t) + V_2(t)} = 1, 
\end{equation}
where 
\begin{align}
V_1(t) ={} & \dfrac{p_1^2 \bar{\xi}^2}{(p_0 t + p_1\bar{\xi})^4} \times p_0t(1-t) + \dfrac{p_0^2 t^2}{(p_0 t + p_1\bar{\xi})^4} \times p_1\bar{\xi}(1 - \bar{\xi}), \label{V.1}\\
V_2(t) ={} & \dfrac{2p_1^2 \bar{\xi}^2}{(p_0 t + p_1\bar{\xi})^4} \sum_{\substack{i<j \\ i,j\in\calH_0}} \Cov(t_i,t_j)  - \dfrac{2p_0p_1t\bar{\xi}}{(p_0 t + p_1\bar{\xi})^4} \sum_{\substack{i \in \calH_0 \\ j \in \calH_1}} \Cov(t_i, t_j) \notag \\
&+ \dfrac{2p_0^2 t^2}{(p_0 t + p_1\bar{\xi})^4} \sum_{\substack{i<j \\ i,j\in\calH_1}} \Cov(t_i,t_j), \label{V.2}
\end{align}
with
\begin{align}
\Cov(t_i,t_j) 
= {} & \int_{|z_i| > |z_{t/2}|} \int_{|z_j| > |z_{t/2}|} \phi(z_i, z_j; \mu_i,\mu_j, 1, 1, \sigma_{ij}) dz_i dz_j \notag \\ 
& - [\Phi(z_{t/2}+\mu_i) + \Phi(z_{t/2}-\mu_i)][\Phi(z_{t/2}+\mu_j) + \Phi(z_{t/2}-\mu_j)], \label{covariance.t}
\end{align}
in which $\phi(z_i, z_j; \mu_i,\mu_j, \sigma_i^2, \sigma_j^2, \sigma_{ij})$ is the probability density function of the bivariate normal distribution with mean $(\mu_i,\mu_j)^T$ and covariance matrix $\left(\begin{matrix} \sigma_i^2 & \sigma_{ij}\\ \sigma_{ij} & \sigma_j^2 \end{matrix}\right)$.
\end{corollary}

The following theorem, which is Theorem 2 in the main article, discusses how the signs of the $\Cov(t_i,t_j)$ is affected by the dependence structure among the test statistics. 

\begin{theorem} \label{Thm-2}
Suppose $(Z_{1},\ldots,Z_{p})^{T} \sim N((\mu_{1},\ldots,\mu_{p})^{T},\bSigma)$ with unit variances, i.e., $\sigma_{jj}=1\text{ for } j=1,\ldots,p$. Assume  $\sigma_{ij} \ne 0$ for the index pair $(i, j), i \ne j$. Then we have:
\begin{itemize}[noitemsep]
\item[$(a)$] for $i ,j\in\calH_0$, $\Cov(t_i, t_j) > 0$;
\item[$(b)$] for $i \in\calH_0$, $j \in\calH_1$, $\Cov(t_i, t_j) < 0$ if $|\mu_j| > 2|z_{t/2}|$ and $t < 2\{1-\Phi^{-1}(\sqrt{\log(3)/2})\}$;
\item[$(c)$] for $i ,j\in\calH_1$, $\Cov(t_i, t_j) > 0$ if $|\sigma_{ij}| \le |\mu|_{\min}/(|\mu|_{\max} + z_{t/2})$ and $\mathrm{sign}(\sigma_{ij}) = \mathrm{sign}(\mu_i \mu_j)$, where $|\mu|_{\min} = \min(|\mu_i|, |\mu_j|)$ and $|\mu|_{\max} = \max(|\mu_i|, |\mu_j|)$.
\end{itemize}
\end{theorem}

These results immediately result in the theorem blow, which gives one set of sufficient conditions that lead to the inflation of the FDP variance; it is Theorem 3 in the main article. 

\begin{theorem} \label{Thm-3}
Suppose $(Z_{1},\ldots,Z_{p})^{T} \sim N((\mu_{1},\ldots,\mu_{p})^{T},\bSigma)$ with $\sigma_{jj}=1\text{ for } j=1,\ldots,p$. For any $i,j\in \mathcal{H}_1, i\neq j$, assume that $|\mu_i| \ge 2|z_{t/2}| > 2\sqrt{\log(3)/2}$,  $|\sigma_{ij}| \le |\mu|_{\min}/(|\mu|_{\max} + z_{t/2})$, and $\mathrm{sign}(\sigma_{ij}) = \mathrm{sign}(\mu_i \mu_j)$, where $|\mu|_{\min} = \min(|\mu_i|, |\mu_j|)$ and $|\mu|_{\max} = \max(|\mu_i|, |\mu_j|)$. Then, $V_2(t)>0$.
\end{theorem}

\section{Lemmas} \label{sec-lemmas}

We first establish the following lemmas; they play fundamental roles in our proofs of the theoretical results given above. In particular, Lemmas \ref{lem-3}--\ref{lem-7} are to support our proof for Theorem \ref{Thm-1}; Lemma \ref{lem-1} is used for Theorem \ref{Thm-2}. 


\begin{lemma}\label{lem-3}
The partial derivatives $f_{t_1, t_2, t_3}^{i_1, i_2, i_3}(a,t)$ can be written to be
$$ f_{t_1, t_2, t_3}^{i_1, i_2, i_3}(a,t) = C^{a}_{t/2}(i_1 + i_2 + i_3)g_{i_1}(x_{t_1})g_{i_2}(x_{t_2})g_{i_3}(x_{t_3}),$$
where
\begin{eqnarray*}
&g_0(x) = 1, \qquad g_1(x) = x, \qquad g_2(x) = x^2 - 1, \qquad g_3(x) = x^3-3x,&\\ & g_4(x) = x^4 - 6x^2 + 3, \qquad  g_5(x) =  x^5 - 10x^3 + 15x& \\
&g_6(x) =  x^6 - 15x^4 + 45x^2 - 15, \qquad  g_7(x) =  x^7 - 21x^5 + 105x^3 - 105x&\\
&g_8(x) =  x^8 - 28x^6 + 210x^4 -420x^2 + 105,&
\end{eqnarray*}
$C^{a}_{t/2}(0) = \Phi(z_{t/2} - \mu_a)$, and $C^{a}_{t/2}(i) = -\phi(z_{t/2}-\mu_a) g_{i-1}(z_{t/2} - \mu_a)$ for $i\geq 1$. Furthermore, for $X\sim N(0,1)$,
\begin{eqnarray*}
E\{g_i(X)\} = 0, && \mbox{ for } i=1,\ldots, 8,\\
\mbox{and } \quad E\{ g_{i_1}(X) g_{i_2}(X)\} = 0, && \mbox{ for  any }i_1, i_2\in \{0, \ldots, 8 \} , i_1 \neq i_2;
\end{eqnarray*}
and for $\mu_a = 0$,
\begin{eqnarray}
C^{a}_{1-t/2}(i) = \begin{cases}
C^{a}_{t/2}(i), &\text{ if i is odd}\\
-C^{a}_{t/2}(i), &\text{ if i is even}.\\ 
\end{cases} \label{eq-lem3-0}
\end{eqnarray}
\end{lemma}

\begin{proof}
The proof is based on straightforward but tedious evaluations on the partial derivatives of $f_{\{t_1, \ldots, t_k\}}(\rho_1, \ldots, \rho_k; a,t)$ with respect to $\rho_1,\ldots, \rho_k$; the details are omitted. 
\end{proof}

\begin{lemma}\label{lem-4}
Recall the definition:  $t_i = \mathrm{1}(|Z_i| > |z_{t/2}|) = \mathrm{1}(P_i < t)$; we have
\begin{itemize}
    \item[(P1). ] when $i\neq j$, and $i,j\in \mathcal{H}_1$,  \begin{eqnarray*}
    &&E(t_it_j) - E(t_i) E(t_j) \\&=& \left\{C^i_{t/2}(1)-C^i_{1-t/2}(1)\right\}\left\{C^j_{t/2}(1)-C^j_{1-t/2}(1)\right\}\sigma_{ij}\\
     &&+ \frac{1}{2}\left\{C^i_{t/2}(2)-C^i_{1-t/2}(2)\right\}\left\{C^j_{t/2}(2)-C^j_{1-t/2}(2)\right\}\sigma_{ij}^2\\
     &&+\frac{E\{g_3^2(X_1)\}}{(3!)^2} \left\{C_{t/2}^i(3) -  C_{1-t/2}^i(3)\right\}\left\{C_{t/2}^j(3) -  C_{1-t/2}^j(3)\right\}\sigma_{ij}^3+O(\sigma_{ij}^4);
    \end{eqnarray*}
    
    \item[(P2).] when $i\in \mathcal{H}_1, j\in \mathcal{H}_0$, for any $t$, there exists a unique root $\mu_t \in (|z_{t/2}|, |z_{t/2}|+1)$ of $C_{t/2}^j(2) - C_{1-t/2}^j(2) = 0$ in $\mu_j$, such that when $\mu_j \leq \mu_t$,
    \begin{eqnarray*}
    && O(\sigma_{ij}^4)< E(t_it_j) - E(t_i) E(t_j) = -\sigma_{ij}^2 \phi(z_{t/2})z_{t/2} \left\{ C_{t/2}^j(2) - C_{1-t/2}^j(2) \right\} + O(\sigma_{ij}^4) \\
    &\leq & \phi(z_{t/2})|z_{t/2}| C_t^{\max} \sigma_{ij}^2 + O(\sigma_{ij}^4),
    \end{eqnarray*}
   where 
   \begin{eqnarray*}
    C_t^{\max} = \sup_{\mu\in (-\mu_t, \mu_t)} \left\{\phi(|z_{t/2}|+ \mu)(|z_{t/2}| + \mu) + \phi(|z_{t/2}| - \mu)(|z_{t/2}| - \mu)\right\};
    \end{eqnarray*} 
    and when  $|\mu_j|\geq \mu_t$,
    \begin{eqnarray*}
    E(t_it_j) - E(t_i) E(t_j) < O(\sigma_{ij}^4).
    \end{eqnarray*}
    
    \item[(P3).] when $i\neq j$, and $i,j\in \mathcal{H}_0$, 
    \begin{eqnarray*}
    E(t_it_j) - E(t_i) E(t_j) = 2 \phi^2(z_{t/2})z_{t/2}^2 \sigma_{ij}^2 + O(\sigma_{ij}^4).
    \end{eqnarray*}
\end{itemize}
\end{lemma}

\begin{proof}
Note that we can write 
\begin{eqnarray*}
E(t_it_j) - E(t_i) E(t_j) = \sum_{a_1, a_2 \in \{1, -1\}}a_1a_2 h_{i,j}(a_1, a_2), \label{eq-lem4-1}
\end{eqnarray*}
where
\begin{eqnarray*}
h_{i,j}(a_1, a_2) = P(Z_i<a_1 z_{t/2}, Z_j < a_2 z_{t/2}) - P(Z_i<a_1 z_{t/2})P(Z_j<a_1 z_{t/2}). \label{eq-lem4-2}
\end{eqnarray*}
We derive $h_{i,j}(a_1, a_2)$ for $a_1 = a_2 = 1$; the derivations for other values of $a_1$ and $a_2$ are similar. Let $X_1, X_2, X_3$ be i.i.d. standard normal random variables, we can write 
\begin{eqnarray*}Z_i = \mu_i + \sqrt{|\sigma_{ij}|} X_1 + \sqrt{1-|\sigma_{ij}|}X_2 &&\\
Z_j = \mu_j + \sqrt{|\sigma_{ij}|} X_1 + \sqrt{1-|\sigma_{ij}|}X_3 &&\mbox{if } \sigma_{ij}\geq 0,
\end{eqnarray*}
and 
\begin{eqnarray*}
Z_j = \mu_j - \sqrt{|\sigma_{ij}|} X_1 + \sqrt{1-|\sigma_{ij}|}X_3 \quad \mbox{if } \sigma_{ij}<0.
\end{eqnarray*} 
Hereafer, we assume that $\sigma_{ij}\geq0$; the case for $\sigma_{ij}<0$ can be derived the same. 
The first term of $h_{i,j}(a_1, a_2)$ is given by 
\begin{equation}
    \begin{split}
        & P(Z_i < z_{t/2}, Z_{j} < z_{t/2}) \\
        = & \int \Phi \left(\dfrac{z_{t/2} - \mu_i - \sqrt{\sigma_{ij}}x_{1}}{\sqrt{1 - \sigma_{ij}}}\right)\Phi \left(\dfrac{z_{t/2} - \mu_j - \sqrt{\sigma_{ij}}x_{1}}{\sqrt{1 - \sigma_{ij}}}\right)\phi(x_1)dx_1.\\
    \end{split} \label{eq-lem4-3}
\end{equation}
Note that by definition $\Phi \left(\dfrac{z_{t/2} - \mu_a - \sqrt{\sigma_{ij}}x_{1}}{\sqrt{1 - \sigma_{ij}}}\right) = f_{\{1\}}(\sqrt{\sigma_{ij}}; a, t)$; applying Taylor expansion and Lemma \ref{lem-3}, we immediately have
\begin{eqnarray}
\Phi \left(\dfrac{z_{t/2} - \mu_a - \sqrt{\sigma_{ij}}x_{1}}{\sqrt{1 - \sigma_{ij}}}\right) = \sum_{i_1 \leq 7} C^a_{t/2}(i_1)g_{i_1}(x_1)(\sqrt{\sigma_{ij}})^{i_1}/(i_1!) + R(\sigma_{ij}), \label{eq-lem4-4}
\end{eqnarray}
where $R(\sigma_{ij})$ is the Lagrange residual term in the Taylor's expansion, satisfying $|R(\sigma_{ij})| \lesssim |f(x_1)| \sigma_{ij}^4$ up to a univeral constant not depending on $x_1$, with $f(x_1)$ being a finite order polynomial function of $x_1$. Combining \eqref{eq-lem4-3} and \eqref{eq-lem4-4} leads to 
\begin{equation}
    \begin{split}
        & P(Z_i < z_{t/2}, Z_{j} < z_{t/2}) \\
        = &\sum_{i_1 + i_2 \leq 7} C^i_{t/2}(i_1)C^j_{t/2}(i_2)(\sqrt{\sigma_{ij}})^{i_1 + i_2}E(g_{i_1}(X_1)g_{i_2}(X_1))/(i_1!i_2!) + O(\sigma_{ij}^4) \\
        = &\sum_{i_1 = 0}^3 C^i_{t/2}(i_1)C^j_{t/2}(i_1)\sigma_{ij}^{i_1}E(g_{i_1}^2(X_1))/(i_1!)^2 + O(\sigma_{ij}^4), \label{eq-lem4-5}\\
    \end{split}
\end{equation}
where the second ``=" is because $E(g_{i_1}(X_1)g_{i_2}(X_1)) = 0$ for $i_1\neq i_2$ based on Lemma \ref{lem-3}; $X_1$ is a standard normal random variable. Furthermore, it is straightforward to check that for $i_1 = 0$,
\begin{eqnarray*}
C^i_{t/2}(i_1)C^j_{t/2}(i_1)\sigma_{ij}^{i_1}E(g_{i_1}^2(X_1))/(i_1!)^2 = P(Z_i < z_{t/2})P(Z_{j} < z_{t/2}),
\end{eqnarray*}
which together with \eqref{eq-lem4-5} leads to
\begin{eqnarray*}
h_{i,j}(1,1) &=& P(Z_i< z_{t/2}, Z_j <  z_{t/2}) - P(Z_i< z_{t/2})P(Z_j< z_{t/2})\\
&=& \sum_{i_1 = 1}^3 C^i_{t/2}(i_1)C^j_{t/2}(i_1)\sigma_{ij}^{i_1}E(g_{i_1}^2(X_1))/(i_1!)^2 + O(\sigma_{ij}^4). 
\end{eqnarray*}
Similarly, we can derive $h_{i,j}(1,-1)$, $h_{i,j}(-1,1)$, and $h_{i,j}(-1,-1)$. As consequence, we have
\begin{equation*}
    \begin{split}
        &E(t_it_j) - E(t_i) E(t_j)= \sum_{a_1, a_2 \in \{1, -1\}}a_1a_2 h_{i,j}(a_1, a_2)\\
        = & \sum_{i_1 = 1}^3 C^i_{t/2}(i_1)C^j_{t/2}(i_1)\sigma_{ij}^{i_1}E\{g_{i_1}^2(X_1)\}/(i_1!)^2\\
        - & \sum_{i_1 = 1}^3 C^i_{1-t/2}(i_1)C^j_{t/2}(i_1)\sigma_{ij}^{i_1}E\{g_{i_1}^2(X_1)\}/(i_1!)^2\\
        - &\sum_{i_1 = 1}^3 C^i_{t/2}(i_1)C^j_{1-t/2}(i_1)\sigma_{ij}^{i_1}E\{g_{i_1}^2(X_1)\}/(i_1!)^2\\
        + &\sum_{i_1 = 1}^3 C^i_{1-t/2}(i_1)C^j_{1-t/2}(i_1)\sigma_{ij}^{i_1}E\{g_{i_1}^2(X_1)\}/(i_1!)^2 + O(\sigma_{ij}^4)\\
        = & \sum_{i_1 = 1}^3 \frac{\sigma_{ij}^{i_1} E\{g_{i_1}^2(X_1)\}}{(i_1!)^2} \left\{C^i_{t/2}(i_1)-C^i_{1-t/2}(i_1)\right\}\left\{C^j_{t/2}(i_1)-C^j_{1-t/2}(i_1)\right\}+O(\sigma_{ij}^4).
\end{split} 
\end{equation*}
We are able to check:
\begin{itemize}

   \item when $\mu_i = 0$, for $i_1 = 1$ or $3$, 
   \begin{eqnarray*}
    C^i_{t/2}(i_1)-C^i_{1-t/2}(i_1) = 0,
    \end{eqnarray*}
    and for $i_1 = 2$ 
     \begin{eqnarray*}
    C^i_{t/2}(i_1)-C^i_{1-t/2}(i_1) = -2 \phi(z_{t/2}) z_{t/2};
    \end{eqnarray*}
    
    \item when $\mu_i \neq 0$, for $i_1 = 1$,
    \begin{eqnarray*}
    C^i_{t/2}(i_1)-C^i_{1-t/2}(i_1) = -\phi(|z_{t/2}|+\mu_i) + \phi(|z_{t/2}|-\mu_i);
    \end{eqnarray*}
    for $i_1 = 2$,
    \begin{eqnarray*}
    C^i_{t/2}(2)-C^i_{1-t/2}(2) &=& -\phi(z_{t/2}- \mu_i)(z_{t/2} - \mu_i) - \phi(z_{t/2} + \mu_i)(z_{t/2} + \mu_i) \\
    &=& \phi(|z_{t/2}|+ \mu_i)(|z_{t/2}| + \mu_i) + \phi(|z_{t/2}| - \mu_i)(|z_{t/2}| - \mu_i).
    \end{eqnarray*}
    We can verify that for any given $t\in (0,1)$, there exists a unique $\mu_t \in (z_{t/2}, z_{t/2}+1)$, which is the root of $C^i_{t/2}(2)-C^i_{1-t/2}(2)=0$ for $\mu_i \in (|z_{t/2}|, |z_{t/2}|+1)$, and $C^i_{t/2}(2)-C^i_{1-t/2}(2) >0 $ if and only if $\mu_i \in (-\mu_t, \mu_t)$. This follows from the following facts and that $\phi(|z_{t/2}|+ \mu_i)(|z_{t/2}| + \mu_i) + \phi(|z_{t/2}| - \mu_i)(|z_{t/2}| - \mu_i)$ as a function of $\mu_i$ is symmetric about $0$.
    \begin{itemize}
        \item[(1)] When $\mu_i \in [-|z_{t/2}|, |z_{t/2}|]$, we have both $|z_{t/2}| + \mu_i>0$ and $|z_{t/2}| - \mu_i>0$; therefore $C^i_{t/2}(2)-C^i_{1-t/2}(2)>0$.
        \item[(2)] When $\mu_i \geq |z_{t/2}|+1$, we have
        \begin{eqnarray*}
        C^i_{t/2}(2)-C^i_{1-t/2}(2) = \phi(u_1)(u_1) - \phi(u_2)(u_2) < 0,
        \end{eqnarray*}
        where $u_1 = \mu_i + |z_{t/2}|$, $u_2 = \mu_i - |z_{t/2}|$, and $u_1 > u_2 > 1$; and we can verify $\phi(u)u$ is straightly decreasing when $u> 1$. 
        \item[(3)] When $\mu_i \in (|z_{t/2}|, |z_{t/2}|+1)$, we have
        \begin{eqnarray*}
        &&\frac{\partial \left\{ C^i_{t/2}(2)-C^i_{1-t/2}(2) \right\}}{\partial \mu_i} \\ 
        &=& \phi(|z_{t/2}|+\mu_i) \left\{ 1 - (|z_{t/2}|+\mu_i)^2 \right\}-\phi(|z_{t/2}|-\mu_i) \left\{ 1 - (|z_{t/2}|-\mu_i)^2 \right\}\\
        &\leq & \phi(|z_{t/2}|+\mu_i)\left\{ (|z_{t/2}|-\mu_i)^2 -  (|z_{t/2}|+\mu_i)^2\right\}\\
        &=& -4\phi(|z_{t/2}|+\mu_i)|z_{t/2}|\mu_i <0,
        \end{eqnarray*}
        where ``$\leq$" is based on the facts that when $\mu_i \in (|z_{t/2}|, |z_{t/2}|+1)$, $1 - (|z_{t/2}-\mu_i)^2>0$ and $\phi(|z_{t/2}|+\mu_i)< \phi(|z_{t/2}|-\mu_i)$. 
    \end{itemize}
    Furthermore, set 
    \begin{eqnarray*}
    C_t^{\max} = \sup_{\mu\in (-\mu_t, \mu_t)} \left\{\phi(|z_{t/2}|+ \mu)(|z_{t/2}| + \mu) + \phi(|z_{t/2}| - \mu)(|z_{t/2}| - \mu)\right\},
    \end{eqnarray*}
    then we have $C^i_{t/2}(2)-C^i_{1-t/2}(2) \leq C_t^{\max}$.

    and for $i_1 = 3$,
    \begin{eqnarray*}
    &&C^i_{t/2}(3)-C^i_{1-t/2}(3) \\ &=& - \phi(z_{t/2}-\mu_i)\left\{ (z_{t/2}-\mu_i)^2 - 1\right\} + \phi(z_{t/2}+\mu_i)\left\{ (z_{t/2}+\mu_i)^2 -1 \right\}\\
    &=&  - \phi(|z_{t/2}|+\mu_i)\left\{ (|z_{t/2}|+\mu_i)^2 - 1\right\} + \phi(|z_{t/2}|-\mu_i)\left\{ (|z_{t/2}|-\mu_i)^2 -1 \right\}
    \end{eqnarray*}
    
    \end{itemize}
    
    We can conclude:
    
    \begin{itemize}

    \item when $\mu_i \neq 0$ and $\mu_j \neq 0$, 
    \begin{eqnarray*}
    &&E(t_it_j) - E(t_i) E(t_j) \\&=& \left\{C^i_{t/2}(1)-C^i_{1-t/2}(1)\right\}\left\{C^j_{t/2}(1)-C^j_{1-t/2}(1)\right\}\sigma_{ij}\\
     &&+ \frac{1}{2}\left\{C^i_{t/2}(2)-C^i_{1-t/2}(2)\right\}\left\{C^j_{t/2}(2)-C^j_{1-t/2}(2)\right\}\sigma_{ij}^2\\
     &&+\frac{E\{g_3^2(X_1)\}}{(3!)^2} \left\{C_{t/2}^i(3) -  C_{1-t/2}^i(3)\right\}\left\{C_{t/2}^j(3) -  C_{1-t/2}^j(3)\right\}\sigma_{ij}^3+O(\sigma_{ij}^4);
    \end{eqnarray*}
    which verifies (P1). 
    
    \item when $\mu_i \neq 0$ but $\mu_j = 0$, 
    \begin{eqnarray*}
    && E(t_it_j) - E(t_i) E(t_j) = -\sigma_{ij}^2 \phi(z_{t/2})z_{t/2} \left\{ C_{t/2}^j(2) - C_{1-t/2}^j(2) \right\} + O(\sigma_{ij}^4) \\
    &\leq & \phi(z_{t/2})|z_{t/2}| C_t^{\max} \sigma_{ij}^2 + O(\sigma_{ij}^4),
    \end{eqnarray*}
    when $\mu_i \in [-\mu_t, \mu_t]$, and 
    \begin{eqnarray*}
    E(t_it_j) - E(t_i) E(t_j)< O(\sigma_{ij}^4),
    \end{eqnarray*}
    otherwise. Therefore, (P2) is proved.

    \item when $\mu_i = 0$ and $\mu_j = 0$,
    \begin{eqnarray*}
    E(t_it_j) - E(t_i) E(t_j) = 2 \phi^2(z_{t/2})z_{t/2}^2 \sigma_{ij}^2 + O(\sigma_{ij}^4),
    \end{eqnarray*}
    which is our (P3) and we complete the proof of this lemma. 
\end{itemize}
\end{proof}

\begin{lemma}\label{lem-6}
Recall the definition:  $t_i = \mathrm{1}(|Z_i| > |z_{t/2}|) = \mathrm{1}(P_i < t)$; we have  
\begin{equation*}
    \begin{split}
        &\sum_{i\neq j\neq k \neq l} \Big\{E(t_i t_j t_k t_l) - E(t_i t_j t_k)E(t_l) - E(t_i t_j t_l)E(t_k) - E(t_i t_k t_l)E(t_j) - E(t_j t_k t_l)E(t_i)\\
        + &E(t_i t_j)E(t_k)E(t_l) + E(t_i t_k)E(t_j)E(t_l) + E(t_i t_l)E(t_j)E(t_k) + E(t_j t_k)E(t_i)E(t_l)\\
        + &E(t_j t_l)E(t_i)E(t_k) + E(t_k t_l)E(t_i)E(t_j) -3E(t_i)E(t_j)E(t_k)E(t_l)\Big\}\\
        = & 3\left[\sum_{i\neq j; i,j\in \mathcal{H}_1}\left\{C_{t/2}^i(1) -  C_{1-t/2}^i(1)\right\}\left\{C_{t/2}^j(1) -  C_{1-t/2}^j(1)\right\}\sigma_{ij}\right]^2 \\  &+ 3\left[\sum_{i\neq j; i,j\in \mathcal{H}_1}\left\{C_{t/2}^i(1) -  C_{1-t/2}^i(1)\right\}\left\{C_{t/2}^j(1) -  C_{1-t/2}^j(1)\right\}\sigma_{ij}\right]\\
    &\times \left[\sum_{i\neq j}\left\{C_{t/2}^i(2) -  C_{1-t/2}^i(2)\right\}\left\{C_{t/2}^j(2) -  C_{1-t/2}^j(2)\right\}\sigma_{ij}^2\right] \\
        & + O\left( \sum_{i\in \mathcal{H}_1}\left( \sum_{j\neq i, j\in\mathcal{H}_1} |\sigma_{ij}|\right)^3 \right) + O\left( \sum_{i\neq j} |\sigma_{ij}| \sum_{k\neq i, k\in \mathcal{H}_1} |\sigma_{ik}|\sum_{l\neq j, l\in \mathcal{H}_1}|\sigma_{jl}| \right)+ O\left( p^2 \sum_{i\neq j}\sigma_{ij}^4\right). \\
    \end{split}
\end{equation*}
\end{lemma}

\begin{proof}
We can write
\begin{equation}
    \begin{split}
        &E(t_i t_j t_k t_l) - E(t_i t_j t_k)E(t_l) - E(t_i t_j t_l)E(t_k) - E(t_i t_k t_l)E(t_j) - E(t_j t_k t_l)E(t_i)\\
        + &E(t_i t_j)E(t_k)E(t_l) + E(t_i t_k)E(t_j)E(t_l) + E(t_i t_l)E(t_j)E(t_k) + E(t_j t_k)E(t_i)E(t_l)\\
        + &E(t_j t_l)E(t_i)E(t_k) + E(t_k t_l)E(t_i)E(t_j) -3E(t_i)E(t_j)E(t_k)E(t_l)\\
        = &  \sum_{a_1, a_2, a_3, a_4 = \{1, -1\}}a_1 a_2 a_3 a_4 h_{i, j, k, l; t}(a_1, a_2, a_3, a_4),
    \end{split} \label{eq-lem6-1}
\end{equation}
where $h_{i, j, k, l}(a_1, a_2, a_3, a_4) = P(Z_i < a_1 z_{t/2}, Z_j < a_2 z_{t/2}, Z_k < a_3 z_{t/2}, Z_l < a_4 z_{t/2}) - P(Z_i < a_1 z_{t/2}, Z_j < a_2 z_{t/2}, Z_k < a_3 z_{t/2})P(Z_l < a_4 z_{t/2}) - P(Z_i < a_1 z_{t/2}, Z_j < a_2 z_{t/2}, Z_l < a_4 z_{t/2})P(Z_k < a_3 z_{t/2}) - P(Z_i < a_1 z_{t/2}, Z_k < a_3 z_{t/2}, Z_l < a_4 z_{t/2})P(Z_j < a_2 z_{t/2}) - P(Z_j < a_2 z_{t/2}, Z_k < a_3 z_{t/2}, Z_l < a_4 z_{t/2})P(Z_i < a_1 z_{t/2}) + P(Z_i < a_1 z_{t/2}, Z_j < a_2 z_{t/2})P(Z_k < a_3 z_{t/2})P(Z_l < a_4 z_{t/2}) + P(Z_i < a_1 z_{t/2}, Z_k < a_3 z_{t/2})P(Z_j < a_2 z_{t/2})P(Z_l < a_4 z_{t/2}) + P(Z_i < a_1 z_{t/2}, Z_l < a_4 z_{t/2})P(Z_j < a_2 z_{t/2})P(Z_k < a_3 z_{t/2}) + P(Z_j < a_2 z_{t/2}, Z_k < a_3 z_{t/2})P(Z_i < a_1 z_{t/2})P(Z_l < a_4 z_{t/2}) + P(Z_j < a_2 z_{t/2}, Z_l < a_4 z_{t/2})P(Z_i < a_1 z_{t/2})P(Z_k < a_3 z_{t/2}) + P(Z_k < a_3 z_{t/2}, Z_l < a_4 z_{t/2})P(Z_i < a_1 z_{t/2})P(Z_j < a_2 z_{t/2}) - 3P(Z_i < a_1 z_{t/2})P(Z_j < a_2 z_{t/2})P(Z_k < a_3 z_{t/2})P(Z_l < a_4 z_{t/2})$. 

Next, we consider the terms in $h_{i,j,k,l}(1,1,1,1)$; those for other values of $a_1,\ldots, a_4$ can be similarly derived. Without loss of generality, we assume $\sigma_{ij}\geq0$ for $i=1,\ldots, p; j=1, \ldots, p$. Note that $h_{i,j,k,l}(1,1,1,1)$  contains 16 terms with the last three terms identical, and its first term is given by
\begin{equation}
    \begin{split}
         & P(Z_i < z_{t/2}, Z_j < z_{t/2}, Z_k < z_{t/2}, Z_l < z_{t/2})\\
         = & \int \Phi \left(\dfrac{z_{t/2} - \mu_i - \sqrt{\sigma_{ij}}x_{1} - \sqrt{\sigma_{ik}}x_{2} - \sqrt{\sigma_{il}}x_{3}}{\sqrt{1 - \sigma_{ij} - \sigma_{ik} - \sigma_{il}}}\right)\Phi\left (\dfrac{z_{t/2} - \mu_j - \sqrt{\sigma_{ij}}x_{1} - \sqrt{\sigma_{jk}}x_{4} - \sqrt{\sigma_{jl}}x_{5}}{\sqrt{1 - \sigma_{ij} - \sigma_{jk} - \sigma_{jl}}}\right)\\
         &\Phi \left(\dfrac{z_{t/2} - \mu_k - \sqrt{\sigma_{ik}}x_{2} - \sqrt{\sigma_{jk}}x_{4} - \sqrt{\sigma_{kl}}x_{6}}{\sqrt{1 - \sigma_{ik} - \sigma_{jk} - \sigma_{kl}}}\right)\Phi\left (\dfrac{z_{t/2} - \mu_l - \sqrt{\sigma_{il}}x_{3} - \sqrt{\sigma_{jl}}x_{5} - \sqrt{\sigma_{kl}}x_{6}}{\sqrt{1 - \sigma_{il} - \sigma_{jl} - \sigma_{kl}}}\right)\\
         &\phi(x_1)\phi(x_2)\phi(x_3)\phi(x_4)\phi(x_5) \phi(x_6)dx_1 dx_2 dx_3 dx_4 dx_5 dx_6.\\
    \end{split} \label{eq-lem6-2}
\end{equation}
Applying Taylor expansion and Lemma \ref{lem-3}, we have
\begin{equation*}
    \begin{split}
        &\Phi \left(\dfrac{z_{t/2} - \mu_i - \sqrt{\sigma_{ij}}x_{1} - \sqrt{\sigma_{ik}}x_{2} - \sqrt{\sigma_{il}}x_{3}}{\sqrt{1 - \sigma_{ij} - \sigma_{ik} - \sigma_{il}}}\right)\\
        = & \sum_{i_1 + i_2 + i_3 \leq 7}C^{i}_{t/2}(i_1 + i_2 + i_3)g_{i_1}(x_1)g_{i_2}(x_2)g_{i_3}(x_3)(\sqrt{\sigma_{ij}})^{i_1}(\sqrt{\sigma_{ik}})^{i_2}(\sqrt{\sigma_{il}})^{i_3}/(i_1 + i_2 + i_3)! + R(\vec{\rho}),
    \end{split}
\end{equation*}
where $R(\vec{\rho})$ is the Lagrange residual term in the Taylor's expansion, and $|R(\vec{\rho})| \lesssim |f(x_1, x_2, x_3)|(|\sigma_{ij}|^4 + |\sigma_{ik}|^4 + |\sigma_{il}|^4)$ up to a universal constant not depending on $x_1, x_2, x_3$, where
$f(x_1, x_2, x_3)$ is a finite order polynomial function of $\{x_1, x_2, x_3\}$. Similarly, we can apply Taylor expansion to other $\Phi(\cdot)$ terms in \eqref{eq-lem6-2}; we obtain
\begin{equation}\label{eq-lem6-3}
    \begin{split}
         & P(Z_i < z_{t/2}, Z_j < z_{t/2}, Z_k < z_{t/2}, Z_l < z_{t/2})\\
         = &  \int \sum_{\sum_{j = 1}^{12} i_j \leq 7}C^{i}_{t/2}(i_1 + i_2 + i_3)C^{j}_{t/2}(i_7 + i_4 + i_5)C^{k}_{t/2}(i_8 + i_{10} + i_6)C^{l}_{t/2}(i_{9} + i_{11} + i_{12})\\
         & g_{i_1}(x_1) g_{i_7}(x_1)g_{i_2}(x_2)g_{i_8}(x_2)g_{i_3}(x_3)g_{i_{9}}(x_3)g_{i_{4}}(x_4)g_{i_{10}}(x_4)g_{i_{5}}(x_5)g_{i_{11}}(x_5)g_{i_{6}}(x_6)g_{i_{12}}(x_6)\\
        &(\sqrt{\sigma_{ij}})^{i_1 + i_7}(\sqrt{\sigma_{ik}})^{i_2 + i_8}(\sqrt{\sigma_{il}})^{i_3 + i_9}(\sqrt{\sigma_{jk}})^{i_4 + i_{10}}(\sqrt{\sigma_{jl}})^{i_5 + i_{11}}(\sqrt{\sigma_{kl}})^{i_6 + i_{12}}\\
        &/((i_1 + i_2 + i_3)!(i_7 + i_4 + i_5)!(i_8 + i_{10} + i_6)!(i_{9} + i_{11} + i_{12})!)\\
        &\phi(x_1)\phi(x_2)\phi(x_3)\phi(x_4)\phi(x_5)\phi(x_6)dx_1 dx_2 dx_3 dx_4 dx_5 dx_6\\
        &+ O\left(|\sigma_{ij}|^4 + |\sigma_{ik}|^4 +|\sigma_{il}|^4 + |\sigma_{jk}|^4 + |\sigma_{jl}|^4 + |\sigma_{kl}|^4\right).
    \end{split}
\end{equation}
Applying Lemma \ref{lem-3}, \eqref{eq-lem6-3} can be further simplified to be
\begin{eqnarray}
&&P(Z_i < z_{t/2}, Z_j < z_{t/2}, Z_k < z_{t/2}, Z_l < z_{t/2}) \nonumber \\
&=& \sum_{\sum_{j=1}^6 i_j \in \{0, 1, 2, 3\} } C^{i}_{t/2}(i_1 + i_2 + i_3)C^{j}_{t/2}(i_1 + i_4 + i_5)C^{k}_{t/2}(i_2 + i_4 + i_6)C^{l}_{t/2}(i_3 + i_5 + i_6) \nonumber \\
&& \hspace{0.4in} \times \frac{\prod_{j=1}^6 E\left\{g_{i_j}^2(X_1)\right\} \sigma_{ij}^{i_1} \sigma_{ik}^{i_2}\sigma_{il}^{i_3}\sigma_{jk}^{i_4}\sigma_{jl}^{i_5} \sigma_{kl}^{i_6}}{(i_1 + i_2 + i_3)!(i_1 + i_4 + i_5)!(i_2 + i_5 + i_6)!(i_3 + i_5 + i_6)!} \nonumber \\
&&+ O\left(|\sigma_{ij}|^4 + |\sigma_{ik}|^4 +|\sigma_{il}|^4 + |\sigma_{jk}|^4 + |\sigma_{jl}|^4 + |\sigma_{kl}|^4\right). \label{eq-lem6-4}
\end{eqnarray}
where $X_1$ is a standard normal random variable. Recall that \eqref{eq-lem6-4} is the first term of $h_{i,j,k,l}(1,1,1,1)$; we next discuss how each subitem under the summation of the main part on the right hand side of \eqref{eq-lem6-4} associates with other terms contributes in the evaluation of \eqref{eq-lem6-1}. In particular, we shall discuss the possible values of $i_1, \ldots, i_6$; clearly if one of them is zero, its corresponding ``$\sigma$" term in  \eqref{eq-lem6-4} does not appear in the expression. 

\begin{itemize}
    \item[(A1)] If the subscripts of the set of $\sigma$'s whose corresponding powers are nonzero contain and only contain three letters, and therefore one letter does not appear in these subscripts, the subitem in \eqref{eq-lem6-4} is identical to the corresponding subitem in one and only one of the second to fifth item in $h_{i,j,k,l}(1,1,1,1)$, whose coefficients are all ``$-$". As a consequence, they cancel each other. 
    
    For example $i_3 = i_5 = i_6 = 0$ but at least two of $\{i_1, i_2, i_4\}$ are nonzero, so that the set of $\sigma$ whose corresponding powers are nonzero is from $\{\sigma_{ij}, \sigma_{ik}, \sigma_{jk}\}$, and $\{i,j,k\}$ all appear in the subscript but $l$ does not. The subitem in \eqref{eq-lem6-4} is given by 
    \begin{eqnarray*}
    &&C^{i}_{t/2}(i_1 + i_2 )C^{j}_{t/2}(i_1 + i_4 )C^{k}_{t/2}(i_2 + i_4 ) C^{l}_{t/2}(0) \nonumber \\
&& \hspace{0.4in} \times \frac{E\left\{g_{i_1}^2(X_1)\right\}E\left\{g_{i_2}^2(X_1)\right\} E\left\{g_{i_4}^2(X_1)\right\} \sigma_{ij}^{i_1} \sigma_{ik}^{i_2}\sigma_{jk}^{i_4}}{(i_1 + i_2 )!(i_1 + i_4 )!(i_2 + i_4)!},
    \end{eqnarray*}
    which is identical and only identical to a corresponding subitem in the Taylor expansion for the second item in $h_{i,j,k,l}(1,1,1,1)$ by noting that $C^{l}_{t/2}(0) = P(Z_l< z_{t/2})$; in particular, similarly to the development of \eqref{eq-lem6-4}, we are able to derive:
    \begin{eqnarray}
&&P(Z_i < z_{t/2}, Z_j < z_{t/2}, Z_k < z_{t/2}) \nonumber \\
&=& \sum_{i_1 + i_2 + i_4 \in \{0, 1, 2, 3\} } C^{i}_{t/2}(i_1 + i_2 )C^{j}_{t/2}(i_1 + i_4 )C^{k}_{t/2}(i_2 + i_4 ) \nonumber \\
&& \hspace{0.4in} \times \frac{E\left\{g_{i_1}^2(X_1)\right\}E\left\{g_{i_2}^2(X_1)\right\} E\left\{g_{i_4}^2(X_1)\right\} \sigma_{ij}^{i_1} \sigma_{ik}^{i_2}\sigma_{jk}^{i_4}}{(i_1 + i_2 )!(i_1 + i_4 )!(i_2 + i_4)!} \nonumber \\
&&+ O\left(|\sigma_{ij}|^4 + |\sigma_{ik}|^4 +|\sigma_{jk}|^4\right). \label{eq-lem6-5}
\end{eqnarray}
    
    \item[(A2)] If the subscripts of the set $\sigma$'s whose corresponding powers are nonzero contain only two letter in their subscripts, indicating that one and one and only of $i_j, j=1,\ldots, 6$ is nonzero, we can find two of the second to fifth terms and one of the sixth to thirteenth terms in $h_{i,j,k,l}(1,1,1,1)$  contain the same corresponding subitem in \eqref{eq-lem6-4}. Note that the coefficients of the second to fifth terms in  $h_{i,j,k,l}(1,1,1,1)$ are ``-", while those of the sixth to thirteenth terms in $h_{i,j,k,l}(1,1,1,1)$ are "$+$", therefore they are cancelled. 
    
    For example, if we have only $i_1>0$, The subitem in \eqref{eq-lem6-4} is given by 
    \begin{eqnarray*}
    C^{k}_{t/2}(0) C^{l}_{t/2}(0)  C^{i}_{t/2}(i_1)C^{j}_{t/2}(i_1) \frac{E\left\{g_{i_1}^2(X_1)\right\} \sigma_{ij}^{i_1}}{(i_1!)^2},
    \end{eqnarray*}
    which is identical to the corresponding terms in $P(Z_i< z_{t/2}, Z_j<z_{t/2}, Z_k<Z_{t/2})P(Z_l<z_{t/2})$ by observing its Taylor expansion given by \eqref{eq-lem6-5}, $P(Z_i< z_{t/2}, Z_j<z_{t/2}, Z_l<Z_{t/2})P(Z_k<z_{t/2})$ (details are omitted), and $P(Z_i< z_{t/2}, Z_j<z_{t/2})P(Z_k<z_{t/2})P(Z_l<Z_{t/2})P(Z_k<z_{t/2})$, since
    \begin{eqnarray*}
    &&P(Z_i < z_{t/2}, Z_j < z_{t/2}) = \sum_{i_1 =0 }^3 \frac{C^{i}_{t/2}(i_1 )C^{j}_{t/2}(i_1 )E\left\{g_{i_1}^2(X_1)\right\} \sigma_{ij}^{i_1} }{(i_1!)^2} + O\left(|\sigma_{ij}|^4\right). \label{eq-lem6-6}
    \end{eqnarray*}
    
    \item[(A3)] If $i_j = 0$ for $j=1,\ldots, 6$, it can be checked that all the terms in $h_{i,j,k,l}(1,1,1,1)$ contain the subitem $C^{k}_{t/2}(0) C^{l}_{t/2}(0)  C^{i}_{t/2}(0)C^{j}_{t/2}(0)$, where seven have the coefficient ``$+$", and the other seven have the coefficient ``$-$", therefore they cancel each other. 
    
\end{itemize}

We observe that all the subitems in the second to the sixteenth items in $h_{i,j,k,l}(1,1,1,1)$ have been considered and therefore cancelled in Cases (A1)--(A3). Only the subitems with the set of $\sigma$'s whose corresponding powers are nonzeros contain all four letters in their subscripts. We consider the following possibilities.
\begin{itemize}
    \item[(B1)] Only two $i_j$'s are nonzero; therefore only two $\sigma$'s appear in the subitem in \eqref{eq-lem6-4}. The corresponding $\sigma$ components in it have three possibilities: $\sigma_{ij}^{i_1}\sigma_{kl}^{i_6}$, $\sigma_{ik}^{i_2}\sigma_{jl}^{i_5}$, and $\sigma_{il}^{i_3} \sigma_{jk}^{i_4}$. These three cases can be considered similarly; as an example, we consider that it is $\sigma_{ij}^{i_1} \sigma_{kl}^{i_6}$, with $i_1>0$ and $i_6>0$. Furthermore since $0< i_1+i_6 \leq 3$, we must have ``$i_1 = 1, i_6 = 1$" or ``$i_1 = 1, i_6 = 2$" or ``$i_1 = 2, i_6 = 1$". 
    
    \begin{itemize}
        \item If $i_1 = 1$ and $i_6 = 1$, the subitem in \eqref{eq-lem6-4} is give by
        \begin{eqnarray*} C^{i}_{t/2}(1)C^{j}_{t/2}(1)C^{k}_{t/2}(1)C^{l}_{t/2}(1)\sigma_{ij}  \sigma_{kl}, \end{eqnarray*}
        and we can check that the corresponding subitem in $h_{i,j,k,l}(a_1, a_2, a_3, a_4)$ is given by 
        \begin{eqnarray*} &&\left\{C^{i}_{t/2}(1)\right\}^{\frac{1+a_1}{2}}\left\{C^{i}_{1-t/2}(1)\right\}^{\frac{1-a_1}{2}}\left\{C^{j}_{t/2}(1)\right\}^{\frac{1+a_2}{2}}\left\{C^{j}_{1-t/2}(1)\right\}^{\frac{1-a_2}{2}}\\
       \times && \left\{C^{k}_{t/2}(1)\right\}^{\frac{1+a_3}{2}}\left\{C^{k}_{1-t/2}(1)\right\}^{\frac{1-a_3}{2}}\left\{C^{l}_{t/2}(1)\right\}^{\frac{1+a_4}{2}}\left\{C^{l}_{1-t/2}(1)\right\}^{\frac{1-a_4}{2}}\sigma_{ij}  \sigma_{kl}.  \end{eqnarray*}
       Using the fact that $C_{t/2}^{a}(1) = C_{1-t/2}^{a}(1)$ when $\mu_a = 0$ given by \eqref{eq-lem3-0}, we can check that when one of $\{\mu_i, \mu_j, \mu_k, \mu_l\}$ is equal to 0, all these subtems will cancel each other across $h_{i,j,k,l}(a_1, a_2, a_3, a_4)$ when evaluating \eqref{eq-lem6-1}. Therefore this subitem is remained only when  $\{\mu_i, \mu_j, \mu_k, \mu_l\}$ are all nonzero, and their contribution in \eqref{eq-lem6-1} is given by 
       \begin{eqnarray*}
       \left\{C_{t/2}^i(1) -  C_{1-t/2}^i(1)\right\}\left\{C_{t/2}^j(1) -  C_{1-t/2}^j(1)\right\}\\\times \left\{C_{t/2}^k(1) -  C_{1-t/2}^k(1)\right\}\left\{C_{t/2}^l(1) -  C_{1-t/2}^l(1)\right\} \sigma_{ij}\sigma_{kl}. 
       \end{eqnarray*}
       
       \item If $i_1=1$, $i_6 =2$, the subitem in \eqref{eq-lem6-4} is given by \begin{eqnarray*} \frac{1}{2}C^{i}_{t/2}(1)C^{j}_{t/2}(1)C^{k}_{t/2}(2)C^{l}_{t/2}(2)\sigma_{ij}  \sigma_{kl}^2, \end{eqnarray*} 
       and the corresponding subitem in $h_{i,j,k,l}(a_1, a_2, a_3, a_4)$ is given by 
        \begin{eqnarray*} &&\left\{C^{i}_{t/2}(1)\right\}^{\frac{1+a_1}{2}}\left\{C^{i}_{1-t/2}(1)\right\}^{\frac{1-a_1}{2}}\left\{C^{j}_{t/2}(1)\right\}^{\frac{1+a_2}{2}}\left\{C^{j}_{1-t/2}(1)\right\}^{\frac{1-a_2}{2}}\\
       \times && \left\{C^{k}_{t/2}(2)\right\}^{\frac{1+a_3}{2}}\left\{C^{k}_{1-t/2}(2)\right\}^{\frac{1-a_3}{2}}\left\{C^{l}_{t/2}(2)\right\}^{\frac{1+a_4}{2}}\left\{C^{l}_{1-t/2}(2)\right\}^{\frac{1-a_4}{2}}\sigma_{ij}  \sigma_{kl}^2.  \end{eqnarray*}
       When $\mu_i = 0$ or $\mu_j = 0$, all these subitems will cancel each other across $h_{i,j,k,l}(a_1, a_2, a_3, a_4)$ when evaluating \eqref{eq-lem6-1}. Therefore this subitem is remained only when  $\mu_i \neq 0$ and $\mu_j \neq 0$, and their contribution in \eqref{eq-lem6-1} is given by 
       \begin{eqnarray*}
       \frac{1}{2}\left\{C_{t/2}^i(1) -  C_{1-t/2}^i(1)\right\}\left\{C_{t/2}^j(1) -  C_{1-t/2}^j(1)\right\}\\\times \left\{C_{t/2}^k(2) -  C_{1-t/2}^k(2)\right\}\left\{C_{t/2}^l(2) -  C_{1-t/2}^l(2)\right\} \sigma_{ij}\sigma_{kl}^2.
       \end{eqnarray*}
       
       \item If $i_1 = 2$ and $i_6 = 1$, we can similarly conclude that the corresponding subitem in  $h_{i,j,k,l}(a_1, a_2, a_3, a_4)$ can be remained only when $\mu_k \neq 0$ and $\mu_l \neq 0$, and their contribution in \eqref{eq-lem6-1} is given by
       \begin{eqnarray*}
       \frac{1}{2}\left\{C_{t/2}^i(2) -  C_{1-t/2}^i(2)\right\}\left\{C_{t/2}^j(2) -  C_{1-t/2}^j(2)\right\}\\\times \left\{C_{t/2}^k(1) -  C_{1-t/2}^k(1)\right\}\left\{C_{t/2}^l(1) -  C_{1-t/2}^l(1)\right\} \sigma_{ij}^2\sigma_{kl}.
       \end{eqnarray*}
       
    \end{itemize}
    
    In summary, for subitems considered in (B1), when summed over $i,j,k,l$ for $i\neq j\neq k \neq l$, they are equal to 
    \begin{eqnarray*}
    3\left[\sum_{i\neq j; i,j\in \mathcal{H}_1}\left\{C_{t/2}^i(1) -  C_{1-t/2}^i(1)\right\}\left\{C_{t/2}^j(1) -  C_{1-t/2}^j(1)\right\}\sigma_{ij}\right]^2 \\ + 3\left[\sum_{i\neq j; i,j\in \mathcal{H}_1}\left\{C_{t/2}^i(1) -  C_{1-t/2}^i(1)\right\}\left\{C_{t/2}^j(1) -  C_{1-t/2}^j(1)\right\}\sigma_{ij}\right]\\
    \times \left[\sum_{i\neq j}\left\{C_{t/2}^i(2) -  C_{1-t/2}^i(2)\right\}\left\{C_{t/2}^j(2) -  C_{1-t/2}^j(2)\right\}\sigma_{ij}^2\right]. 
    \end{eqnarray*}

\item[(B2)] There are three $i_j$'s are nonzero, and therefore each is equal to 1. Note that the corresponding subscripts of $\sigma$ have six letters; and each of $i,j,k,l$ must appear at least once. There are two possibilities
\begin{itemize}
    \item One letter appears three times, but each of the other letters appear once. For example, letter $i$ appears three times, but $j,k,l$ appear once, namely $i_1 = i_2 = i_3 = 1$, which corresponds to the subitem in \eqref{eq-lem6-4}:
    \begin{eqnarray*}
    C_{t/2}^i(3) C_{t/2}^j(1) C_{t/2}^k(1) C_{t/2}^l(1)\frac{\sigma_{ij} \sigma_{ik}\sigma_{il}}{3!}.
    \end{eqnarray*}
    Using similar arguments as those in (B1), we can conclude that this subitem is remained in $h_{i,j,k,l}(a_1, a_2, a_3, a_4)$ only when $\mu_i, \mu_j, \mu_k, \mu_l$ are all nonzero, and they are in the order of $O(\sigma_{ij}\sigma_{ik}\sigma_{il})$. 
    
    \item Two letters appear twice, but each of the other two letters appear once. For example, letter $i$ and $j$ appear twice, but $k,l$ appear once, namely $i_1 = i_2 = i_5 = 1$, which corresponds to the subitem in \eqref{eq-lem6-4}:
    \begin{eqnarray*}
    C_{t/2}^i(2) C_{t/2}^j(2) C_{t/2}^k(1) C_{t/2}^l(1)\frac{\sigma_{ij} \sigma_{ik}\sigma_{jl}}{2!2!}.
    \end{eqnarray*}
    Using similar arguments as those in (B1), we can conclude that this subitem is remained in $h_{i,j,k,l}(a_1, a_2, a_3, a_4)$ only when $\mu_k \neq 0$ and $\mu_l \neq 0$, and they are in the order of $O(\sigma_{ij} \sigma_{ik}\sigma_{jl})$. 
    
\end{itemize}

In summary, the subitems considered in (B2), when summed over $i,j,k,l$ for $i\neq j\neq k\neq l$, they are in the order of 
\begin{eqnarray*}
O\left( \sum_{i\in \mathcal{H}_1}\left( \sum_{j\neq i, j\in\mathcal{H}_1} |\sigma_{ij}|\right)^3 \right) + O\left( \sum_{i\neq j} |\sigma_{ij}| \sum_{k\neq i, k\in \mathcal{H}_1} |\sigma_{ik}|\sum_{l\neq j, l\in \mathcal{H}_1}|\sigma_{jl}| \right). 
\end{eqnarray*}
    
\end{itemize}

\end{proof}

\begin{lemma}\label{lem-5}
Recall the definition: $t_i = \mathrm{1}(|Z_i| > |z_{t/2}|) = \mathrm{1}(P_i < t)$; we have
\begin{equation*}
    \begin{split}
        &\sum_{i\neq j\neq k}\Big|E(t_i t_j t_k) - E(t_i t_j)E(t_k) - E(t_i t_k)E(t_j) - E(t_j t_k)E(t_i) + 2E(t_i)E(t_j)E(t_k)\Big|\\
        = & O\left(\sum_{i=1}^p \left(\sum_{j\neq i, j\in \mathcal{H}_1} |\sigma_{ij}|\right)^2 \right) + O\left(\sum_{i\in \mathcal{H}_1}\left( \sum_{j\neq i} |\sigma_{ij}|\right) \left( \sum_{k\neq i, k\in \mathcal{H}_1}|\sigma_{ik}|\right) \right) \\ &+  O\left(\sum_{i\neq j\neq k} |\sigma_{ij}\sigma_{ik} \sigma_{kl}| \right) + O\left (p \sum_{i\neq j} \sigma_{ij}^4 \right). \\
    \end{split}
\end{equation*}
\end{lemma}

\begin{proof}
The proof of this lemma is similar to but simpler than that of Lemma \ref{lem-6}. The details are thus omitted. 
\end{proof}

\begin{lemma}\label{lem-7}
Let
\begin{eqnarray*}
m(\bar{V},\bar{R}) = \frac{\bar{V}}{E(\bar{R})} - \frac{E(\bar{V})}{[E(\bar{R})]^2} \bar{R}. 
\end{eqnarray*}
Assume $\limsup_{p\to \infty} p_0t/(p_1\bar \xi) < 1$, and as $p$ is sufficiently large,  
\begin{eqnarray}
&& \sum_{i,j}|\sigma_{ij}| = O\left(p^{2-\delta}\right), \label{eq-lem7-assum-0}\\
&&\sum_{i\neq j; i,j\in \mathcal{H}_0} \sigma_{ij}^2 \geq \frac{C_t^{\max}}{\phi(z_{t/2})|z_{t/2}|} \sum_{i\in \mathcal{H}_1, j\in \mathcal{H}_0, \mu_i\in[-\mu_t, \mu_t]} \sigma_{ij}^2, \label{eq-lem7-assum-1-added} \\
&&\sum_{i\neq j; i,j\in \mathcal{H}_0} \sigma_{ij}^2 + p \gtrsim \sum_{i\neq j; i, j\in \mathcal{H}_1} \sigma_{ij}^2  \label{eq-lem7-assum-1-added-2}\\
&& \sum_{i\neq j} \sigma_{ij}^4 = o\left(\sum_{i\neq j; i,j\in \mathcal{H}_0} \sigma_{ij}^2  + p_0\right),  \label{eq-lem7-assum-3}
\end{eqnarray}
where $\mu_t$ and $C_t^{\max}$ are defined in Lemma \ref{lem-4}. 
We have
\begin{eqnarray}
E\left[\{R-E(R)\}^4\right] &=& o\left(p^4\mbox{Var}(m(\bar V, \bar R)\right) \label{eq-lem7-0-1}\\
E\left[\{V-E(V)\}^4\right] &=& o\left(p^4\mbox{Var}(m(\bar V, \bar R)\right).  \label{eq-lem7-0-2}
\end{eqnarray}
\end{lemma}

\begin{proof} 
We consider $\mbox{Var}\left(m(\bar{V},\bar{R})\right) $ first. Recall that $R = V + S$; therefore $\bar R = \bar V + \bar S$ with $\bar R = R/p$, $\bar V = V/p$ and $\bar S = S/p$, and 
\begin{eqnarray*}
m(\bar{V},\bar{R}) &=& \frac{E(\bar S)}{\{E(\bar R)\}^2} \bar V - \frac{E(\bar V)}{\{E(\bar R)\}^2} \bar S\\
&=& A \left( \frac{\bar V}{E(\bar V)} - \frac{\bar S}{E(\bar S)} \right),
\end{eqnarray*}
where $A = E(\bar S) E(\bar V)/\{E(\bar R)\}^2$. We shall derive
\begin{eqnarray}
\mbox{Var}\left\{m(\bar{V},\bar{R})\right\} \gtrsim \frac{\sum_{i\neq j; i,j\in \mathcal{H}_0} \sigma_{ij}^2 + p_0  + \mbox{Var}(S)}{p^2}.  \label{eq-lem7-1}
\end{eqnarray}
To this end, 
consider 
\begin{eqnarray}
\frac{1}{A^2}\mbox{Var}\left(m(\bar{V},\bar{R})\right) &=& \mbox{Var}(\bar V/E(\bar V)) +  \mbox{Var}(\bar S/E(\bar S)) - 2\mbox{Cov}(\bar V/E(\bar V), \bar S/E(\bar S)), \label{eq-lem7-1-added-6}
\end{eqnarray}
and based on Lemma \ref{lem-4},
\begin{eqnarray}
&&\mbox{Cov}(\bar V/E(\bar V), \bar S/E(\bar S)) = \frac{1}{E(V) E(S)} \sum_{i\in \mathcal{H}_1, j\in \mathcal{H}_0} \left\{ E(t_it_j) - E(t_i)E(t_j) \right\} \nonumber\\
&\leq& \frac{\phi( z_{t/2})|z_{t/2}| C_t^{\max}}{p_0 p_1 t \bar \xi}  \sum_{i\in \mathcal{H}_1, j\in \mathcal{H}_0; \mu_i \in[-\mu_t, \mu_t]} \sigma_{ij}^2 + \frac{1}{p_0 p_1 t \bar \xi} O\left(\sum_{i\in \mathcal{H}_1, j\in \mathcal{H}_0} \sigma_{ij}^4\right).  \label{eq-lem7-1-added-7}
\end{eqnarray}
Combining \eqref{eq-lem7-1-added-6} with \eqref{eq-lem7-1-added-7}, and applying Lemma \ref{lem-4}, we have
\begin{eqnarray*}
\frac{1}{A^2}\mbox{Var}\left(m(\bar{V},\bar{R})\right) &\geq& \frac{2\phi^2(z_{t/2})z_{t/2}^2 \sum_{i\neq j; i,j\in \mathcal{H}_0} \sigma_{ij}^2 + p_0t(1-t) + O\left( \sum_{i\neq j, i,j\in \mathcal{H}_0} \sigma_{ij}^4\right)}{p_0^2 t^2}\\
&&-\frac{p_0 t}{p_1 \bar \xi}\frac{2 \phi(z_{t/2}) |z_{t/2}|C_t^{\max}}{p_0^2  t^2}  \sum_{i\in \mathcal{H}_1, j\in \mathcal{H}_0; \mu_i \in [-\mu_t, \mu_t]} \sigma_{ij}^2 - \frac{O\left(\sum_{i\in \mathcal{H}_1, j\in \mathcal{H}_0} \sigma_{ij}^4\right)}{p_0p_1t\bar \xi} \\&&+ \mbox{Var}(S/E(S))\\
&\gtrsim& \frac{\sum_{i\neq j; i,j\in \mathcal{H}_0} \sigma_{ij}^2 + p_0}{p_0^2} + \mbox{Var}(S/E(S)),
\end{eqnarray*}
where  ``$\gtrsim$" is based on assumptions \eqref{eq-lem7-assum-1-added} and \eqref{eq-lem7-assum-3}; this verifies \eqref{eq-lem7-1}.

We next consider $E[\{R-E(R)\}^4]$:
\begin{equation}\label{eq-lem7-2}
    \begin{split}
         &E[\{R-E(R)\}^4] =  E(R^4) - 4E(R^3)E(R) + 6E(R^2)E^2(R) - 3E^4(R)\\
        = & E\left\{\left(\sum_{i=1}^p t_i\right)^4\right\} - 4E\left\{\left(\sum_{i=1}^p t_i\right)^3\right\}E\left(\sum_{i=1}^p t_i\right) \\ &+ 6E\left\{\left(\sum_{i=1}^p t_i\right)^2\right\}E^2\left(\sum_{i=1}^p t_i\right) - 3E^4\left(\sum_{i=1}^p t_i\right)\\
        = & E\left(\sum_{i, j, k, l = 1}^p t_i t_j t_k t_l\right) - 4E\left(\sum_{i, j, k = 1}^pt_it_jt_k\right)E\left(\sum_{l=1}^pt_l\right) \\ &+ 6E\left(\sum_{i, j = 1}^p t_it_j\right)E\left(\sum_{k=1}^pt_k\right)E\left(\sum_{l=1}^pt_l\right)\\
        & - 3E\left(\sum_{i= 1}^p t_i\right)E\left(\sum_{j= 1}^p t_j\right)E\left(\sum_{k= 1}^p t_k\right)E\left(\sum_{l= 1}^p t_l\right). 
    \end{split}
\end{equation}
We need to consider the following possibilities:
\begin{itemize}
    \item[(C1)] Based on Lemma \ref{lem-5}, the collection of terms in (\ref{eq-lem7-2}) where $\{i, j, k, l\}$ are all different is given by 
    \begin{equation}
        \begin{split}
            & \sum_{i\neq j\neq k \neq l}\Big\{E(t_i t_j t_k t_l) - E(t_i t_j t_k)E(t_l) - E(t_i t_j t_l)E(t_k) - E(t_i t_k t_l)E(t_j) - E(t_j t_k t_l)E(t_i)\\
        + &E(t_i t_j)E(t_k)E(t_l) + E(t_i t_k)E(t_j)E(t_l) + E(t_i t_l)E(t_j)E(t_k) + E(t_j t_k)E(t_i)E(t_l)\\
        + &E(t_j t_l)E(t_i)E(t_k) + E(t_k t_l)E(t_i)E(t_j) -3E(t_i)E(t_j)E(t_k)E(t_l)\Big\}\\
        = &3\left[\sum_{i\neq j; i,j\in \mathcal{H}_1}\left\{C_{t/2}^i(1) -  C_{1-t/2}^i(1)\right\}\left\{C_{t/2}^j(1) -  C_{1-t/2}^j(1)\right\}\sigma_{ij}\right]^2 \\ & + 3\left[\sum_{i\neq j; i,j\in \mathcal{H}_1}\left\{C_{t/2}^i(1) -  C_{1-t/2}^i(1)\right\}\left\{C_{t/2}^j(1) -  C_{1-t/2}^j(1)\right\}\sigma_{ij}\right]\\
    & \times \left[\sum_{i\neq j}\left\{C_{t/2}^i(2) -  C_{1-t/2}^i(2)\right\}\left\{C_{t/2}^j(2) -  C_{1-t/2}^j(2)\right\}\sigma_{ij}^2\right] \\
        &+ O\left( \sum_{i\in \mathcal{H}_1}\left( \sum_{j\neq i, j\in\mathcal{H}_1} |\sigma_{ij}|\right)^3 \right) + O\left( \sum_{i\neq j} |\sigma_{ij}| \sum_{k\neq i, k\in \mathcal{H}_1} |\sigma_{ik}|\sum_{l\neq j, l\in \mathcal{H}_1}|\sigma_{jl}| \right)+ O\left( p^2 \sum_{i\neq j}\sigma_{ij}^4\right)\\
        &\equiv 3\mathcal{I}_1 + 3\mathcal{I}_2 + \mathcal{I}_3 + \mathcal{I}_4 + \mathcal{I}_5. 
        \end{split} \label{eq-lem7-3}
    \end{equation}
    We consider the above $\mathcal{I}$ terms one by one.  We consider $\mathcal{I}_1+\mathcal{I}_2$ first. Note that 
    \begin{eqnarray*}
    \mathcal{I}_1 = \mathcal{I}_{1,1} \cdot \mathcal{I}_{1,1} \\
    \mathcal{I}_2 = \mathcal{I}_{1,1} \cdot \mathcal{I}_{2,2},
    \end{eqnarray*}
    with
    \begin{eqnarray*}
    \mathcal{I}_{1,1} &=& \sum_{i\neq j; i,j\in \mathcal{H}_1}\left\{C_{t/2}^i(1) -  C_{1-t/2}^i(1)\right\}\left\{C_{t/2}^j(1) -  C_{1-t/2}^j(1)\right\}\sigma_{ij}\\
    \mathcal{I}_{2,2} &=& \sum_{i\neq j}\left\{C_{t/2}^i(2) -  C_{1-t/2}^i(2)\right\}\left\{C_{t/2}^j(2) -  C_{1-t/2}^j(2)\right\}\sigma_{ij}^2.
    \end{eqnarray*}
    As a consequence
    \begin{eqnarray}
    \mathcal{I}_1 + \mathcal{I}_2 = \mathcal{I}_{1,1}(\mathcal{I}_{1,1} + \mathcal{I}_{2,2}), \label{eq-lem7-4-added}
    \end{eqnarray}
    and 
    we need to verify
    \begin{eqnarray}
        \mathcal{I}_1 + \mathcal{I}_2  = o\left(p^4\mbox{Var}(\bar V, \bar R) \right) \label{eq-lem7-4}
        \end{eqnarray}
    Based on the assumption \eqref{eq-lem7-assum-0}, we immediately have 
    \begin{eqnarray}
    \mathcal{I}_{1,1} = o(p^2), \label{eq-lem7-5}
    \end{eqnarray}
    and we have the following decomposition for $\mathcal{I}_{2,2} + \mathcal{I}_{1,1}$:
    \begin{eqnarray}
    \mathcal{I}_{2,2} + \mathcal{I}_{1,1} &=&  \left\{\mathcal{I}_{1,1} +\frac{1}{2}\mathcal{I}_{2,2} + \mathcal{I}_{2,3} + p_0t(1-t) +  \sum_{i=1}^{p_1} \xi_i (1-\xi_i)\right\} \nonumber \\&& + \left\{\frac{1}{2}\mathcal{I}_{2,2} -\mathcal{I}_{2,3} - p_0t(1-t) -  \sum_{i=1}^{p_1} \xi_i (1-\xi_i)\right\}, \label{eq-lem7-6}
    \end{eqnarray}
    where 
    \begin{eqnarray*}
    \mathcal{I}_{2,3} = \frac{E\{g_3^2(X_1)\}}{(3!)^2}\sum_{i \neq j} \left\{C_{t/2}^i(3) -  C_{1-t/2}^i(3)\right\}\left\{C_{t/2}^j(3) -  C_{1-t/2}^j(3)\right\}\sigma_{ij}^3.  \label{eq-lem7-7}
    \end{eqnarray*}
    Furthermore, with Lemmas \ref{lem-3} and \ref{lem-4}, we are able to check 
    \begin{eqnarray}
    \mbox{Var}(R) &=& \sum_{i\neq j} \{E(t_it_j) - E(t_i)E(t_j)\} + \sum_{i=1}^p \left[E(t_i^2) - \{E(t_i)\}^2\right] \nonumber \\
    &=& \mathcal{I}_{1,1} +\frac{1}{2}\mathcal{I}_{2,2} + \mathcal{I}_{2,3} + p_0t(1-t) +  \sum_{i=1}^{p_1} \xi_i (1-\xi_i) + O\left(\sum_{i\neq j} \sigma_{ij}^4 \right), \label{eq-lem7-8}
    \end{eqnarray}
    and
    \begin{eqnarray}
    &&\mbox{Var}(V) + 2 \mbox{Cov}(V, S) \nonumber \\&=& \sum_{i\neq j; i,j\in \mathcal{H}_0} \{E(t_it_j) - E(t_i)E(t_j)\} \nonumber \\&& + \sum_{i\in \mathcal{H}_1, j\in \mathcal{H}_0} \{E(t_it_j) - E(t_i)E(t_j)\} + \sum_{i=1}^{p_0} \left[E(t_i^2) - \{E(t_i)\}^2\right]\nonumber \\
    &=& \frac{1}{2}\mathcal{I}_{2,2} + p_0t(1-t) \nonumber \\ && - \frac{1}{2}\sum_{i\neq j; i,j\in \mathcal{H}_1} \left\{C_{t/2}^i(2) -  C_{1-t/2}^i(2)\right\}\left\{C_{t/2}^j(2) -  C_{1-t/2}^j(2)\right\}\sigma_{ij}^2 \nonumber \\
    && + O\left(\sum_{i\neq j}\sigma_{ij}^4 \right)\label{eq-lem7-9}
    \end{eqnarray}
    Combining \eqref{eq-lem7-6}, \eqref{eq-lem7-8}, and  \eqref{eq-lem7-9} leads to 
    \begin{eqnarray}
    \mathcal{I}_{2,2} + \mathcal{I}_{1,1} &=& \mbox{Var}(R) + \mbox{Var}(V) + 2 \mbox{Cov}(V,S) - 2p_0t(1-t) - \sum_{i=1}^{p_1} \xi_i (1-\xi_i)\nonumber\\&&
    - I_{2,3} + \frac{1}{2} \sum_{i\neq j; i,j\in \mathcal{H}_1} \left\{C_{t/2}^i(2) -  C_{1-t/2}^i(2)\right\}\left\{C_{t/2}^j(2) -  C_{1-t/2}^j(2)\right\}\sigma_{ij}^2\nonumber\\
    && + O\left(\sum_{i\neq j}\sigma_{ij}^4 \right)\nonumber\\
    &=& 2\mbox{Var}(V) + 4\mbox{Cov}(V, S) +  \mbox{Var}(S) - 2p_0t(1-t) - \sum_{i=1}^{p_1} \xi_i (1-\xi_i)\nonumber\\&&
    - I_{2,3} + \frac{1}{2} \sum_{i\neq j; i,j\in \mathcal{H}_1} \left\{C_{t/2}^i(2) -  C_{1-t/2}^i(2)\right\}\left\{C_{t/2}^j(2) -  C_{1-t/2}^j(2)\right\}\sigma_{ij}^2\nonumber \\
    && + O\left(\sum_{i\neq j}\sigma_{ij}^4 \right). \label{eq-lem7-10}
    \end{eqnarray}
    Based on \eqref{eq-lem7-1} and Cauchy-Schwartz inequality, we conclude 
    \begin{equation}
    2\mbox{Var}(V) + 4\mbox{Cov}(V, S) +  \mbox{Var}(S) - 2p_0t(1-t) - \sum_{i=1}^{p_1} \xi_i (1-\xi_i) = O\left(p^2\mbox{Var}(m(\bar V, \bar R)\right); \label{eq-lem7-11}
    \end{equation}
    and based on the assumptions \eqref{eq-lem7-assum-1-added-2} and \eqref{eq-lem7-assum-3}, we have
    \begin{equation}
    \frac{1}{2} \sum_{i\neq j; i,j\in \mathcal{H}_1} \left\{C_{t/2}^i(2) -  C_{1-t/2}^i(2)\right\}\left\{C_{t/2}^j(2) -  C_{1-t/2}^j(2)\right\}\sigma_{ij}^2
    + O\left(\sum_{i\neq j}\sigma_{ij}^4 \right) = O\left(p^2\mbox{Var}(m(\bar V, \bar R)\right) \label{eq-lem7-12}
    \end{equation}
    Combining \eqref{eq-lem7-10} \eqref{eq-lem7-11} and \eqref{eq-lem7-12}, we have
    \begin{eqnarray*}
     \mathcal{I}_{2,2} + \mathcal{I}_{1,1} &=& -\mathcal{I}_{2,3} + O\left(p^2\mbox{Var}(m(\bar V, \bar R)\right),
    \end{eqnarray*}
    which together with \eqref{eq-lem7-5} leads to 
    \begin{eqnarray}
    \mathcal{I}_{1,1}(\mathcal{I}_{2,2} + \mathcal{I}_{1,1}) = -\mathcal{I}_{1,1}\mathcal{I}_{2,3} + o\left(p^4\mbox{Var}(m(\bar V, \bar R)\right). \label{eq-lem7-13}
    \end{eqnarray}
    For $\mathcal{I}_{1,1}\mathcal{I}_{2,3}$:
    \begin{eqnarray*}
    |\mathcal{I}_{1,1}\mathcal{I}_{2,3}| \lesssim \sum_{i\neq j; i,j\in \mathcal{H}_1}|\sigma_{i,j}|\sum_{k\neq l}|\sigma_{k,l}|^3 \lesssim p^2 \sum_{i\neq j} \sigma_{ij}^4 = o\left(p^4\mbox{Var}(m(\bar V, \bar R)\right),
    \end{eqnarray*}
    which together with \eqref{eq-lem7-13} and \eqref{eq-lem7-4-added} verifies \eqref{eq-lem7-4}. 

  We proceed to consider $\mathcal{I}_3$. 
  \begin{eqnarray}
  \sum_{i\neq j\neq k\neq l; i,j,k,l\in \mathcal{H}_1}|\sigma_{ij}||\sigma_{ik}||\sigma_{il}|&\leq& \sum_{k\ne l, k,l\in \mathcal{H}_1}\left(\sum_{i\neq j, i,j\in \mathcal{H}_1} \sigma_{ij}^2\right)^{1/2}\left(\sum_{i\neq j\neq k\neq l; i,j\in \mathcal{H}_1}\sigma_{ik}^2 \sigma_{il}^2\right)^{1/2}\nonumber\\
  &=& \left(\sum_{i\neq j; i,j\in \mathcal{H}_1} \sigma_{ij}^2\right)^{1/2} \sum_{k\neq l; k,l\in \mathcal{H}_1}\left\{(p_1-3) \sum_{i\neq k\neq l; i\in \mathcal{H}_1}(\sigma_{ik}^2 \sigma_{il}^2)\right\}^{1/2}\nonumber \\
  &\leq& p_1(p_1-3)^{1/2}\left(\sum_{i\neq j; i,j\in \mathcal{H}_1} \sigma_{ij}^2\right)^{1/2} \left\{ \sum_{k\neq l\neq i; i,k,l\in \mathcal{H}_1} (\sigma_{ik}^2 \sigma_{il}^2) \right\}^{1/2}\nonumber \\
  &\lesssim & p_1^{3/2}\left(\sum_{i\neq j; i,j\in \mathcal{H}_1} \sigma_{ij}^2\right)^{1/2} \left( p_1\sum_{i\neq k; i,k\in \mathcal{H}_1} \sigma_{ik}^4 + p\sum_{i\neq l; i,l\in \mathcal{H}_1} \sigma_{il}^4 \right)^{1/2}\nonumber\\
  &=& p_1^2 \left(\sum_{i\neq j; i,j\in \mathcal{H}_1} \sigma_{ij}^2\right)^{1/2} \left(2\sum_{i\neq j; i,j\in \mathcal{H}_1}\sigma_{ij}^4\right)^{1/2}, \label{eq-lem7-14-added-1}
  \end{eqnarray}
  where the first ``$\leq$" is because of the Cauchy–Schwarz inequality, the second ``$\leq$" is derived from the Jensen's inequality by noting that $\sqrt{x}$ is a concave function, the ``$\lesssim$" is based on $\sigma_{ik}^2\sigma_{il}^2 \lesssim \sigma_{ik}^4 + \sigma_{il}^4$. The far right hand side of \eqref{eq-lem7-14-added-1}
  is in the order of $o\left(p^4\mbox{Var}(m(\bar V, \bar R)\right)$ based on \eqref{eq-lem7-1} and the assumptions \eqref{eq-lem7-assum-1-added-2} and \eqref{eq-lem7-assum-3}; this implies  
  \begin{eqnarray}
  \mathcal{I}_3 = O\left( \sum_{i\in \mathcal{H}_1}\left( \sum_{j\neq i, j\in\mathcal{H}_1} |\sigma_{ij}|\right)^3 \right) = o\left(p^4\mbox{Var}(m(\bar V, \bar R)\right). \label{eq-lem7-14}
  \end{eqnarray}
 Similarly to the development of \eqref{eq-lem7-14-added-1}, we can obtain
  \begin{eqnarray}
  \mathcal{I}_4 = O\left( \sum_{i\neq j} |\sigma_{ij}| \sum_{k\neq i, k\in \mathcal{H}_1} |\sigma_{ik}|\sum_{l\neq j, l\in \mathcal{H}_1}|\sigma_{jl}| \right) = o\left(p^4\mbox{Var}(m(\bar V, \bar R)\right).  \label{eq-lem7-15}
  \end{eqnarray}
  Last, 
  \begin{eqnarray}
  \mathcal{I}_5 = O\left( p^2 \sum_{i\neq j}\sigma_{ij}^4\right)  =  o\left(p^4\mbox{Var}(m(\bar V, \bar R)\right) \label{eq-lem7-16}
  \end{eqnarray}
  is implied by \eqref{eq-lem7-1} and the assumption \eqref{eq-lem7-assum-3}. 
  
  Combining \eqref{eq-lem7-4}, \eqref{eq-lem7-14}, \eqref{eq-lem7-15}, and \eqref{eq-lem7-16},  we conclude that the collection of terms in \eqref{eq-lem7-2} where $\{i,j,k,l\}$ are all different is in the order of $o\left(p^4\mbox{Var}(m(\bar V, \bar R)\right)$. 
  
    \item[(C2)] If one and only one pair in $\{i, j, k, l\}$ are equal to each other, as an example, we consider $i\neq j \neq k=l$, the collection of terms in (\ref{eq-lem7-2}) is given by
    \begin{equation}
        \begin{split}
            & \sum_{i\neq j\neq k}\Big\{E(t_it_j t_k) - E(t_it_k)E(t_j) - E(t_jt_k)E(t_i) - 2E(t_i t_j t_k)E(t_k)
        + E(t_i)E(t_j)E(t_k) \\& + 2E(t_i t_k)E(t_j)E(t_k) + 2E(t_j t_k)E(t_i)E(t_k) + E(t_i t_j)E^2(t_k) -3E^2(t_k)E(t_i)E(t_j)\Big\}\\
        = & \sum_{i\neq j \neq k} \Big[\big\{E(t_i t_j t_k) - E(t_i t_k)E(t_j) - E(t_j t_k)E(t_i) - E(t_i t_j)E(t_k) \\ & \hspace{0.6in} + 2E(t_i)E(t_j)E(t_k)\big\}\{1 - 2E(t_k)\}\Big] \\ &+ \sum_{i\neq j \neq k}\big\{E(t_k) - E^2(t_k)\big\}\big\{E(t_i t_j) - E(t_i)E(t_j)\big\}\\
        =& O\left(\sum_{i=1}^p \left(\sum_{j\neq i, j\in \mathcal{H}_1} |\sigma_{ij}|\right)^2 \right) + O\left(\sum_{i\in \mathcal{H}_1}\left( \sum_{j\neq i} |\sigma_{ij}|\right) \left( \sum_{k\neq i, k\in \mathcal{H}_1}|\sigma_{ik}|\right) \right) \\ &+  O\left(\sum_{i\neq j\neq k} |\sigma_{ij}\sigma_{ik} \sigma_{kl}| \right) + O\left (p \sum_{i\neq j} \sigma_{ij}^4 \right) + O\left(p \sum_{i\neq j}\sigma_{ij}^2\right) + O\left(p \sum_{i\neq j; i,j \in \mathcal{H}_1} |\sigma_{ij}|\right)\\
        =& \mathcal{I}_6 + \mathcal{I}_7 + \mathcal{I}_8 + \mathcal{I}_9 + \mathcal{I}_{10} + \mathcal{I}_{11},
        \end{split} \label{eq-lem7-17}
    \end{equation}
    where to achieve the second ``$=$", we have applied Lemmas \ref{lem-4} and \ref{lem-5}. Next, we consider these ``$\mathcal{I}$" terms one by one. 
    For $\mathcal{I}_6$, since $|\sigma_{ij}|<1$, therefore by noting \eqref{eq-lem7-1} and the assumption \eqref{eq-lem7-assum-0}, we have
    \begin{eqnarray}
    \mathcal{I}_6 = O\left(\sum_{i=1}^p \left(\sum_{j\neq i, j\in \mathcal{H}_1} |\sigma_{ij}|\right)^2 \right)\lesssim p_1 \sum_{j\neq i, j\in \mathcal{H}_1}|\sigma_{ij}| = p_1O(p^{2-\delta})= o\left(p^4\mbox{Var}(m(\bar V, \bar R)\right);
    \end{eqnarray}
    and similarly
    \begin{eqnarray}
    \mathcal{I}_7 = O\left(\sum_{i\in \mathcal{H}_1}\left( \sum_{j\neq i} |\sigma_{ij}|\right) \left( \sum_{k\neq i, k\in \mathcal{H}_1}|\sigma_{ik}|\right) \right) \lesssim p \sum_{i\neq k; i,k \in \mathcal{H}_1} |\sigma_{ik}| = o\left(p^4\mbox{Var}(m(\bar V, \bar R)\right). 
    \end{eqnarray}
    For $\mathcal{I}_8$, since $|\sigma_{ij}|<1$, we have
    \begin{eqnarray}
    \mathcal{I}_8 = O\left(\sum_{i\neq j\neq k} |\sigma_{ij}\sigma_{ik} \sigma_{kl}| \right) \lesssim p \sum_{i\neq j} |\sigma_{ij}| = o\left(p^4\mbox{Var}(m(\bar V, \bar R)\right).
    \end{eqnarray}
    For $\mathcal{I}_9$ and $\mathcal{I}_{10}$, clearly
    \begin{eqnarray}
    \mathcal{I}_9 \lesssim \mathcal{I}_{10} = o\left(p^4\mbox{Var}(m(\bar V, \bar R))\right);
    \end{eqnarray}
    and 
    \begin{eqnarray}
    \mathcal{I}_{11} = O\left(p \sum_{i\neq j; i,j \in \mathcal{H}_1} |\sigma_{ij}|\right) = o\left(p^4\mbox{Var}(m(\bar V, \bar R))\right) \label{eq-lem7-18}
    \end{eqnarray}
    is obtained based on the assumption \eqref{eq-lem7-assum-0}. 
    
    Combining \eqref{eq-lem7-17}--\eqref{eq-lem7-18}, we conclude that the collection of terms in \eqref{eq-lem7-2} where one and only one pair in  $\{i,j,k,l\}$ are equal to each other is in the order of $o\left(p^4\mbox{Var}(m(\bar V, \bar R)\right)$.
    
    \item[(C3)] The collection of terms in \eqref{eq-lem7-2} where two pairs in 
    $\{i, j, k, l\}$ are equal to each other carries only $O(p^2)$ terms; then based on \eqref{eq-lem7-1}, we conclude that this collection of terms is in the order of $o\left(p^4\mbox{Var}(m(\bar V, \bar R)\right)$. 
    
    \item[(C4)] The collection of terms in \eqref{eq-lem7-2} where at least three of $\{i, j, k, l\}$ are mutually equal to each other also carries $O(p^2)$ terms, therefore this collection of terms is  in the order of $o\left(p^4\mbox{Var}(m(\bar V, \bar R)\right)$.  

\end{itemize}

In summary, the discussion of (C1)--(C4) implies 
\begin{eqnarray*}
E[\{R-E(R)\}^4] = o\left(p^4\mbox{Var}(m(\bar V, \bar R))\right),
\end{eqnarray*}
which complete the proof for \eqref{eq-lem7-0-1}. 

The proof for \eqref{eq-lem7-0-2} is similar but simpler. In particular, similar developments in Lemmas \ref{lem-6} and \ref{lem-7} can be applied to derive the rates of convergence for the terms in $E(V-E(V))^4$.  
But we observe that
since $V$ is the summation of $t_i$ over $i\in \mathcal{H}_0$, the corresponding terms involving $\mathcal{H}_1$ in both Lemmas will not appear, which  simplifies the developments. Here we skip the details for this development to avoid lengthy presentations.  
\end{proof}

\begin{lemma} \label{lem-1}
Suppose $(Z_{1},\ldots,Z_{p})^{T} \sim N((\mu_{1},\ldots,\mu_{p})^{T},\bSigma)$ with unit variances, i.e., $\sigma_{jj}=1\text{ for } j=1,\ldots,p$. For any pair of indices $(i, j)$ with $i \ne j$,
\begin{align*}
\frac{\partial\Cov(t_i, t_j)}{\partial\sigma_{ij}} =  \frac{1}{2\pi\sqrt{1-\sigma_{ij}^2}} & \left[ \exp \left\{-\frac{\mu_{i-}^2 + \mu_{j-}^2 - 2\sigma_{ij}\mu_{i-}\mu_{j-}}{2(1-\sigma_{ij}^2)} \right\} + \exp \left\{-\frac{\mu_{i+}^2 + \mu_{j+}^2 - 2\sigma_{ij}\mu_{i+}\mu_{j+}}{2(1-\sigma_{ij}^2)} \right\} \right. \\
 & \left. - \exp \left\{-\frac{\mu_{i-}^2 + \mu_{j+}^2 - 2\sigma_{ij}\mu_{i-}\mu_{j+}} {2(1-\sigma_{ij}^2)} \right\} - \exp \left\{-\frac{\mu_{i+}^2 + \mu_{j-}^2 - 2\sigma_{ij} \mu_{i+}\mu_{j-}} {2(1-\sigma_{ij}^2)} \right\} \right],
\end{align*}
where $\mu_{i-} = \mu_i + z_{t/2}$ and $\mu_{i+} = \mu_i - z_{t/2}$ for $i = 1,\ldots,p$.
\end{lemma}

\begin{proof}
Let $R = (-\infty, z_{t/2}) \cup (-z_{t/2}, \infty)$ and $R-\mu_i = (-\infty, z_{t/2}-\mu_i) \cup (-z_{t/2}-\mu_i, \infty)$, then
\begin{align*}
& P(|Z_i| > |z_{t/2}|, |Z_j| > |z_{t/2}|) \\ 
= {} & \int_{R} \phi(z_i - \mu_i) \int_{R} \frac{1}{\sqrt{1 - \sigma_{ij}^2}} \phi \left\{\frac{(z_j - \mu_j) - \sigma_{ij}(z_i - \mu_i)}{\sqrt{1 - \sigma_{ij}^2}} \right\} dz_j dz_i \\
= {} & \int_{R} \phi(z_i - \mu_i) \left[ \Phi \left\{\frac{- \mu_{j+} - \sigma_{ij}(z_i - \mu_i)}{\sqrt{1 - \sigma_{ij}^2}} \right\} + \Phi \left\{\frac{\mu_{j-} + \sigma_{ij}(z_i - \mu_i)}{\sqrt{1 - \sigma_{ij}^2}} \right\} \right] dz_i \\
= {} & \int_{R-\mu_i} \phi(z_i) \left[ 1 - \Phi \left(\frac{ \mu_{j+} + \sigma_{ij}z_i}{\sqrt{1 - \sigma_{ij}^2}} \right) + \Phi \left(\frac{\mu_{j-} + \sigma_{ij}z_i}{\sqrt{1 - \sigma_{ij}^2}} \right) \right] dz_i
\end{align*}
where the first equality holds because the conditional distribution of $Z_j$ given $Z_i = z_i$ is $N[\mu_j + \sigma_{ij} (z_i - \mu_i), 1-\sigma_{ij}^2]$. Then,
\begin{align*}
\frac{\partial \Cov(t_i, t_j)}{\partial \sigma_{ij}} = {} & \frac{\partial P(|Z_i| > |z_{t/2}|, |Z_j| > |z_{t/2}|)}{\partial \sigma_{ij}}\\
= {} & - (1-\sigma_{ij}^2)^{-3/2} \int_{R-\mu_i} \phi(z_i) \phi\left(\frac{\mu_{j+} + \sigma_{ij}z_i}{\sqrt{1 - \sigma_{ij}^2}} \right) (\mu_{j+} \sigma_{ij} + z_i) dz_i \notag\\
& + (1-\sigma_{ij}^2)^{-3/2} \int_{R - \mu_i} \phi(z_i) \phi\left(\frac{\mu_{j-} + \sigma_{ij}z_i}{\sqrt{1 - \sigma_{ij}^2}} \right) (\mu_{j-} \sigma_{ij} + z_i) dz_i \\
= {} & - \frac{1}{2 \pi} (1-\sigma_{ij}^2)^{-3/2} \int_{R - \mu_i} \exp\left\{-\frac{z_{i}^{2} + 2 \sigma_{ij} z_{i} \mu_{j+} + \mu_{j+}^{2}}{2(1-\sigma_{ij}^{2})} \right\} (\mu_{j+} \sigma_{ij} + z_i) dz_i \\
& + \frac{1}{2 \pi} (1-\sigma_{ij}^2)^{-3/2} \int_{R - \mu_i} \exp\left\{-\frac{z_{i}^{2} + 2 \sigma_{ij} z_{i} \mu_{j-} + \mu_{j-}^{2}}{2(1-\sigma_{ij}^{2})} \right\} (\mu_{j-} \sigma_{ij} + z_i) dz_i\\
= {} & \frac{1}{2 \pi \sqrt{1-\sigma_{ij}^2}} \left[ \exp\left\{-\frac{\mu_{i+}^2 + \mu_{j+}^2 - 2 \sigma_{ij} \mu_{i+}\mu_{j+}}{2(1-\sigma_{ij}^{2})}\right\} - \exp\left\{-\frac{\mu_{i-}^2 + \mu_{j+}^2 - 2 \sigma_{ij} \mu_{i-}\mu_{j+}}{2(1-\sigma_{ij}^{2})}\right\} \right. \\
 & \left. - \exp\left\{-\frac{\mu_{i+}^2 + \mu_{j-}^2 - 2 \sigma_{ij} \mu_{i+}\mu_{j-}}{2(1-\sigma_{ij}^{2})}\right\} + \exp\left\{-\frac{\mu_{i-}^2 + \mu_{j-}^2 - 2 \sigma_{ij} \mu_{i-}\mu_{j-}}{2(1-\sigma_{ij}^{2})}\right\}\right].
\end{align*}
\end{proof}

\section{Proof of Theorem \ref{Thm-1} and Corollary \ref{Corrollary-asym-var}} \label{sec-proof-thm-1}

Define $\bar{V} = V(t)/p$ and $\bar{R} = R(t)/p$, then $\FDP(t) = \bar{V}/\bar{R}$. Let $H(v,r)= v/r$ be a function of $(v,r)$ and apply Taylor expansion on this function at $(E(\bar V), E(\bar R))$ and plugging in $v = \bar V$, $r = \bar R \vee c$ to the expansion, where $\bar R\vee c = \max\{\bar R, c\}$ and $c$ is a sufficiently small constant with $0<c<0.5E(\bar R)$; we have
\begin{eqnarray*}
\frac{\bar V}{\bar R\vee c} = \frac{E(\bar{V})}{E(\bar{R})} + \frac{\bar{V} - E(\bar{V})}{E(\bar{R})} - \frac{E(\bar{V})}{\{E(\bar{R})\}^2}\Big\{\bar{R}\vee c - E(\bar{R})\Big\} + r^*(\bar{V}, \bar{R}),  \label{eq-thm1-1}
\end{eqnarray*}
where $r^*(\bar V, \bar R)$ is the remainder term in the Taylor expansion. Since $E(\bar R) > c$ and $\bar R\vee c\geq c$, and in an arbitrary order of the partial derivatives of the function $H(v,r)$, only $r$ can appear in the denominator, therefore, we can verify 
\begin{eqnarray}
|r^*(\bar V, \bar R)| \lesssim \left\{\bar V - E(\bar V)\right\}^2 + \left\{\bar R \vee c - E(\bar R)\right\}^2. \label{eq-thm1-2}
\end{eqnarray}
Therefore, we can have
\begin{eqnarray}
\mbox{FDP}(t) = E(\bar V)/E(\bar R) +  m(\bar{V},\bar{R}) + \mathcal{J} +r^*(\bar V, \bar R), \label{eq-thm1-3}
\end{eqnarray}
where
\begin{eqnarray*}
m(\bar{V},\bar{R}) &=&  \frac{\bar{V}}{E(\bar{R})} - \frac{E(\bar{V})}{\{E(\bar{R})\}^2} \bar{R} \label{eq-thm1-4}\\
\mathcal{J} &=& \frac{\bar V}{\bar R} - \frac{\bar V}{\bar R \vee c}.  \label{eq-thm1-5}
\end{eqnarray*}
Define $r(\bar V, \bar R) = r^*(\bar V, \bar R) + \mathcal{J}$. To complete the proof of the theorem, it suffices to show that both $E(\mathcal{J}^2)$ and $E\{r^{*2}(\bar V, \bar R)\}$ can be dominated by $\mbox{Var}\{m(\bar{V},\bar{R})\}$.  We consider $E\{r^{*2}(\bar V, \bar R)\}$ first. Based on \eqref{eq-thm1-1}, we have
\begin{eqnarray*}
|r^*(\bar V, \bar R)| &\lesssim& \left\{\bar V - E(\bar V)\right\}^2 + \left\{\bar R \vee c - \bar R\right\}^2 + \left\{ \bar R -  E(\bar R)\right\}^2\\
&=&\left\{\bar V - E(\bar V)\right\}^2 + \left\{ c - \bar R\right\}^2\mathrm{1}(\bar R < c)  + \left\{ \bar R -  E(\bar R)\right\}^2\\
&\leq & \left\{\bar V - E(\bar V)\right\}^2 + c^2\mathrm{1}(\bar R < c)  + \left\{ \bar R -  E(\bar R)\right\}^2, \label{eq-thm1-6}
\end{eqnarray*}
which leads to 
\begin{eqnarray}
 E\left\{r^{*2}(\bar V, \bar R)\right\}
&\lesssim& E\left\{\bar V - E(\bar V)\right\}^4 + c^4P(\bar R < c)  + E\left\{ \bar R -  E(\bar R)\right\}^4 \nonumber \\
&\leq & E\left\{\bar V - E(\bar V)\right\}^4   + 2  E\left\{ \bar R -  E(\bar R)\right\}^4, \label{eq-thm1-7}
\end{eqnarray}
since based on the definition of $c\in (0, 0.5 E(\bar R))$, we have
\begin{eqnarray*}
P(\bar R<c) \leq P(|\bar R - E(\bar R)|> c) \leq \frac{E\left\{\bar R - E(\bar R)\right\}^4}{c^4}. \label{eq-thm1-8}
\end{eqnarray*}
Combining \eqref{eq-thm1-7} with Lemma \ref{lem-7}, we conclude
\begin{eqnarray}
E\{r^{*2}(\bar V, \bar R)\} = o\left(\mbox{Var}\{m(\bar{V},\bar{R})\}\right). \label{eq-thm1-9}
\end{eqnarray}

Next, we consider $E(\mathcal{J}^2)$; in particular
\begin{eqnarray*}
\mathcal{J} = \frac{\bar V}{\bar R} - \frac{\bar V}{\bar R \vee c} = \frac{\bar V}{\bar R}\cdot \frac{\bar R\vee c - \bar R}{\bar R \vee c} = \frac{\bar V}{\bar R} \cdot \frac{c-\bar R}{c}\cdot \mathrm{1}(\bar R<c).
\end{eqnarray*}
By noting $\bar V/\bar R\in (0,1)$ and $(c-\bar R)/c \in (0,1)$ when $\bar R<c$, we have
\begin{eqnarray}
E(\mathcal{J}^2) \leq E\{\mathrm{1}(\bar R<c)\} \leq \frac{E\left\{\bar R - E(\bar R)\right\}^4}{c^4} = o\left(\mbox{Var}\{m(\bar{V},\bar{R}))\right). \label{eq-thm1-10}
\end{eqnarray}
Combining \eqref{eq-thm1-9}, \eqref{eq-thm1-10} with \eqref{eq-thm1-3}, we complete the proof of Theorem \ref{Thm-1}.

We proceed to show Corollary \ref{Corrollary-asym-var}. From (\ref{eq-thm1-3}), we have that
\begin{eqnarray*}
\mbox{Var} \left\{ \mbox{FDP}(t) \right\} &=& \mbox{Var} \left\{ m(\bar{V},\bar{R}) + \mathcal{J} + r^*(\bar V, \bar R) \right\} \\
&=& \mbox{Var} \left\{ m(\bar{V},\bar{R}) \right\} + \mbox{Var} \left\{ \mathcal{J} \right\} + \mbox{Var} \left\{r^*(\bar V, \bar R) \right\}\\
&& + 2 \mbox{Cov} \left\{ m(\bar{V},\bar{R}), \mathcal{J} \right\} + 2 \mbox{Cov} \left\{ m(\bar{V},\bar{R}), r^*(\bar V, \bar R) \right\}+ 2 \mbox{Cov} \left\{ \mathcal{J}, r^*(\bar V, \bar R) \right\},
\end{eqnarray*}
which together with \eqref{eq-thm1-9}, \eqref{eq-thm1-10}, and Cauchy-Schwarz inequality immediately leads to
\begin{equation*}
\lim_{p \to \infty} \frac{\mbox{Var} \left\{ \mbox{FDP}(t) \right\}}{\mbox{Var} \left\{ m(\bar{V},\bar{R}) \right\}} = 1.
\end{equation*}
With straightforward evaluation, we can check that the denominator above is identical to $V_1(t) + V_2(t)$. We complete the proof of Corollary \ref{Corrollary-asym-var}.

\section{Proof of Theorem \ref{Thm-2}} \label{sec-proof-thm-2}

\subsubsection*{Part (a)}

When $\mu_i = \mu_j = 0$, according to Lemma \ref{lem-1}, 
\[ \frac{\partial \Cov(t_i, t_j)}{\partial \sigma_{ij}} =  \frac{1}{\pi \sqrt{1-\sigma_{ij}^2}} \left\{\exp\left(-\frac{z_{t/2}^{2}}{1+\sigma_{ij}} \right)  - \exp\left(-\frac{z_{t/2}^{2}}{1-\sigma_{ij}} \right) \right\}. \]
Thus, $\partial \Cov(t_i, t_j)/\partial \sigma_{ij} > 0$ if $\sigma_{ij} > 0$ and $\partial \Cov(t_i, t_j)/\partial \sigma_{ij} < 0$ if $\sigma_{ij} < 0$. In addition, $\Cov(t_i, t_j) = 0$ when $\sigma_{ij} = 0$. Then $\Cov(t_i, t_j) > 0$ for $\sigma_{ij} \ne 0$.

\subsubsection*{Part (b)}

When $\mu_i = 0$ and $\mu_j \ne 0$, based on Lemma \ref{lem-1}, $\partial \Cov(t_i, t_j)/\partial \sigma_{ij}$ is an odd function of $\sigma_{ij}$ and an even function of $\mu_j$. As $\Cov(t_i, t_j) = 0$ when $\sigma_{ij} = 0$, it suffices to show that $\partial \Cov(t_i, t_j)/\partial \sigma_{ij} < 0$ when $\sigma_{ij} > 0$ and $\mu_j > 0$.

According to Lemma \ref{lem-1}, up to positive multiplicative factors, 
\begin{align}
\frac{\partial\Cov(t_i, t_j)}{\partial\sigma_{ij}} \propto {} & \exp \left\{-\frac{\mu_{j-}^2 - 2\sigma_{ij}z_{t/2}\mu_{j-}}{2(1-\sigma_{ij}^2)} \right\} + \exp \left\{-\frac{\mu_{j+}^2 + 2\sigma_{ij}z_{t/2}\mu_{j+}}{2(1-\sigma_{ij}^2)} \right\} \notag \\
 & - \exp \left\{-\frac{\mu_{j+}^2 - 2\sigma_{ij}z_{t/2}\mu_{j+}} {2(1-\sigma_{ij}^2)} \right\} - \exp \left\{-\frac{\mu_{j-}^2 + 2\sigma_{ij}z_{t/2}\mu_{j-}} {2(1-\sigma_{ij}^2)} \right\} \notag \\
 \propto {} & \exp \left\{-\frac{(1-\sigma_{ij})z_{t/2}\mu_j-\sigma_{ij}z_{t/2}^2}{1-\sigma_{ij}^2} \right\} + \exp \left\{-\frac{-(1-\sigma_{ij})z_{t/2}\mu_j-\sigma_{ij}z_{t/2}^2}{1-\sigma_{ij}^2} \right\} \notag \\
 & - \exp \left\{-\frac{-(1+\sigma_{ij})z_{t/2}\mu_j+\sigma_{ij}z_{t/2}^2} {1-\sigma_{ij}^2} \right\} - \exp \left\{-\frac{(1+\sigma_{ij})z_{t/2}\mu_j+\sigma_{ij}z_{t/2}^2} {1-\sigma_{ij}^2} \right\} \notag \\
 \le {} & \exp \left\{-\frac{2z_{t/2}\mu_j-\sigma_{ij}z_{t/2}\mu_j}{2(1-\sigma_{ij}^2)} \right\} + \exp \left\{-\frac{-2z_{t/2}\mu_j+3\sigma_{ij}z_{t/2}\mu_j}{2(1-\sigma_{ij}^2)} \right\} \notag \\
 & - \exp \left\{-\frac{-2z_{t/2}\mu_j-3\sigma_{ij}z_{t/2}\mu_j} {2(1-\sigma_{ij}^2)} \right\} - \exp \left\{-\frac{2z_{t/2}\mu_j+\sigma_{ij}z_{t/2}\mu_j} {2(1-\sigma_{ij}^2)} \right\}, \label{eq-thm2-1}
\end{align}
where the last inequality holds as $\mu_j > 2 |z_{t/2}|$. Let 
\[a = \exp[-\sigma_{ij}z_{t/2}\mu_j/\{2(1-\sigma_{ij}^2)\}] \quad\text{and}\quad  b = \exp\{-z_{t/2}\mu_j/(1-\sigma_{ij}^2)\},\]
then the expression in (\ref{eq-thm2-1}) becomes
\begin{align*}
a^{-1}b + a^3 b^{-1} - a^{-3}b^{-1} - ab &= a^{-3} b^{-1} (a^2 b^2 + a^6 - a^4 b^2 - 1) \\
&= a^{-3} b^{-1} (a^2 - 1) (a^4 + a^2 + 1 - a^2 b^2).
\end{align*}
Since $a > 1$, $\partial\Cov(t_i, t_j)/\partial\sigma_{ij} < 0$ is implied by $a^4 + a^2 + 1 - a^2 b^2 < 0$, which is further implied by $3a^2 < b^2$. It is straightforward to show that $t < 2\{1-\Phi^{-1}(\sqrt{\log(3)/2})\}$ and $\mu_j > 2 |z_{t/2}|$ are sufficient conditions for $3a^2 < b^2$.

\subsubsection*{Part (c)}

As $\mathrm{sign}(\sigma_{ij}) = \mathrm{sign}(\mu_i \mu_j)$, we can assume $\mu_i > 0$, $\mu_j > 0$, and $\sigma_{ij} > 0$ in the proof, because for the other possible cases, we can add a negative sign to $Z_i$ and/or $Z_j$, which does not actually change $\Cov(t_i, t_j)$ and thus does not affect the conclusion. Similar to Part (b), It suffices to show that $\partial \Cov(t_i, t_j)/\partial \sigma_{ij} > 0$ when $\mu_i > 0$, $\mu_j > 0$, and $\sigma_{ij} > 0$. Further, we assume $\mu_i < \mu_j$ without loss of generality. 

According to Lemma \ref{lem-1}, up to a positive multiplicative factor, 
\begin{align}
\frac{\partial\Cov(t_i, t_j)}{\partial\sigma_{ij}} \propto {} & \exp\left(\frac{z_{t/2}^2 \sigma_{ij}}{1-\sigma_{ij}^2}\right) \left[ \exp \left\{-\frac{z_{t/2}(\mu_i+\mu_j)(1-\sigma_{ij})}{1-\sigma_{ij}^2} \right\} + \exp \left\{\frac{z_{t/2}(\mu_i+\mu_j)(1-\sigma_{ij})}{1-\sigma_{ij}^2} \right\} \right] \notag \\
 & - \exp\left(-\frac{z_{t/2}^2 \sigma_{ij}}{1-\sigma_{ij}^2}\right) \left[ \exp \left\{-\frac{z_{t/2}(\mu_i-\mu_j)(1+\sigma_{ij})}{1-\sigma_{ij}^2} \right\} + \exp \left\{\frac{z_{t/2}(\mu_i-\mu_j)(1+\sigma_{ij})}{1-\sigma_{ij}^2} \right\} \right] \label{eq-thm2-2}
\end{align}
Introducing the following notation
\[a = \exp\left(\frac{z_{t/2}^2 \sigma_{ij}}{1-\sigma_{ij}^2}\right), \  x = \exp \left\{-\frac{z_{t/2}(\mu_i+\mu_j)(1-\sigma_{ij})}{1-\sigma_{ij}^2} \right\}, \  y = \exp \left\{\frac{z_{t/2}(\mu_i-\mu_j)(1+\sigma_{ij})}{1-\sigma_{ij}^2} \right\},\]
where $a > 1$, $x > 1$, and $y > 1$. Then, the expression in (\ref{eq-thm2-2}) becomes $a(x + x^{-1}) - a^{-1}(y + y^{-1})$.

Next, we show that $a x > a^{-1} y$ implies that $a(x + x^{-1}) > a^{-1}(y + y^{-1})$. When $x > y$, it is easy to see that $a(x + x^{-1}) > a^{-1}(y + y^{-1})$. When $x \le y$ and $a x > a^{-1} y$, $a(x + x^{-1}) > a^{-1} y + a y^{-1} \ge a^{-1}(y + y^{-1})$. 

Therefore, a sufficient condition for $\partial \Cov(t_i, t_j)/\partial \sigma_{ij} > 0$ is $a x > a^{-1} y$, which is equivalent to $\sigma_{ij} < \mu_i/(\mu_j + z_{t/2})$.